\documentclass[11pt,reqno]{amsart}
\usepackage{amsfonts,amsmath,amssymb}
\newtheorem{thm}{Theorem}[section]
\newtheorem{proposition}[thm]{Proposition}
\newtheorem{corollary}[thm]{Corollary}
\newtheorem{lemma}[thm]{Lemma}
\newtheorem{definition}[thm]{Definition}
\newtheorem{example}[thm]{Example}
\newtheorem{remark}[thm]{Remark}

\usepackage{color}
\usepackage{graphicx}
\usepackage{epstopdf}

\title[Spherical functions and deformed quantum CMS problem]{Symmetric Lie superalgebras and deformed quantum Calogero-Moser problems}

\author{ A.N. Sergeev}\address{Department of Mathematics, Saratov State University
 Astrakhanskaya 83, Saratov 410012, Russia 
and National Research University Higher School of Economics, 
20 Myasnitskaya Ulitsa, Moscow 101000, Russia }
\email{SergeevAN@info.sgu.ru}
\author{A.P. Veselov}
\address{Department of Mathematical Sciences,
Loughborough University, Loughborough LE11 3TU, UK  and Moscow State University, Moscow 119899, Russia}
\email{A.P.Veselov@lboro.ac.uk}
\begin{document}
\maketitle

\begin{abstract}
The representation theory of symmetric Lie superalgebras and corresponding spherical functions are studied in relation with the theory
of the deformed quantum Calogero-Moser systems.
In the special case of symmetric pair $\frak g=\mathfrak{gl}(n,2m), \frak k=\mathfrak{osp}(n,2m)$ we establish a natural bijection between projective covers of spherically typical irreducible $\frak g$-modules and the finite dimensional generalised eigenspaces of the algebra of Calogero-Moser integrals $\frak{D}_{n,m}$ acting on the corresponding Laurent quasi-invariants $\frak A_{n,m}$.
\end{abstract}

\tableofcontents

\section{Introduction} 

In 1964 Berezin et al \cite{BPF} made an important remark that the radial part of the Laplace--Beltrami operator on the symmetric space $X=SL(n)/SO(n)$ 
$$
L=\Delta+\sum_{i<j}^n\coth(x_i-x_j) (\partial_i-\partial_j)
$$
is conjugated to the quantum Hamiltonian
$$
H=\Delta+ \sum_{i<j}^n\frac{1}{2\sinh^2(x_i-x_j)}
$$
describing the pairwise interacting particle on the line. 

This was probably the first recorded observation of the connection between the theory of symmetric spaces and the theory of what later became known as Calogero-Moser, or Calogero-Moser-Sutherland (CMS), integrable models \cite{CMS}. Olshanetsky and Perelomov suggested a class of generalisations of CMS systems related to any root system and showed that the radial parts of all irreducible symmetric spaces 
are conjugated to some particular operators from this class \cite{OP3}. The joint eigenfunctions of the corresponding commutative algebras of quantum integrals are zonal spherical functions. In the $A_n$ case this leads to an important notion of the Jack polynomials introduced by H. Jack independently around the same time \cite{Jack}.

The discovery of the Dunkl operator technique led to an important link of the CMS systems with the representation theory of Cherednik algebras,
see Etingof's lectures \cite{Etingof}. 

It turned out  that there are  other integrable generalisations, which have only partial symmetry and called deformed CMS systems  \cite{CFV}.
Their relation with symmetric superspaces was first discovered by one of the authors in \cite{Ser} and led to a class of such operators related to the basic classical Lie superalgebras, which was introduced in \cite{SV}.

In this paper we develop this link further to study the representation theory of symmetric Lie superalgebras and the related spherical functions.
Such Lie superalgebra is a pair $(\frak g,\theta),$ where $\frak g$ is a Lie superalgebra and  $\theta$ is an involutive automorphism of $\frak g$. It corresponds to the symmetric pair $X=(\frak g, \frak k),$ where $\frak k$ is $\theta$-invariant part of $\frak g$ and can be considered as an algebraic version of the symmetric superspace $G/K$.  

In the particular case of $X=(\frak {gl}(n,2m), \frak {osp}(n,2m))$ the radial part of the corresponding  Laplace-Beltrami operator in the exponential coordinates 
is a particular case of the deformed CMS operator related to 
Lie superalgebra $\mathfrak{gl}(n,m)$ \cite{SV}
$$ {\mathcal L}=\sum_{i=1}^n
\left(x_{i}\frac{\partial}{\partial
x_{i}}\right)^2+k\sum_{j=1}^m
\left(y_{j}\frac{\partial}{\partial
y_{j}}\right)^2-k \sum_{i < j}^n
\frac{x_{i}+x_{j}}{x_{i}-x_{j}}\left(
x_{i}\frac{\partial}{\partial x_{i}}-
x_{j}\frac{\partial}{\partial x_{j}}\right)$$ 
\begin{equation}
\label{Lrad}
 -\sum_{i <
j}^m \frac{y_{i}+y_{j}}{y_{i}-y_{j}}\left(
y_{i}\frac{\partial}{\partial y_{i}}-
y_{j}\frac{\partial}{\partial y_{j}}\right)-\sum_{i=1 }^n\sum_{
j=1}^m \frac{x_{i}+y_{j}}{x_{i}-y_{j}}\left(
x_{i}\frac{\partial}{\partial x_{i}}-k
y_{j}\frac{\partial}{\partial y_{j}}\right)
\end{equation} 
corresponding to the special value of parameter $k=-\frac12.$
According to \cite{SV} it has infinitely many commuting differential operators 
generating the algebra of quantum deformed CMS integrals $\frak{D}_{n,m}$.

We study the action of $\frak{D}_{n,m}$ on the algebra $\frak A_{n,m}$ of $S_n\times S_m$-invariant Laurent polynomials
$f\in\Bbb C[x_1^{\pm1},\dots, x_n^{\pm1},y_1^{\pm1}\dots,y_m^{\pm1}]^{S_n\times S_m}$
satisfying the quasi-invariance condition
\begin{equation}
\label{quasi}
x_i\frac{\partial f}{\partial x_i}-ky_j\frac{\partial f}{\partial y_j}\equiv 0
\end{equation}
on the hyperplane $x_i=y_j$ for all $i=1,\dots,n$, $j=1,\dots,m$ with $k=-\frac12$. 
It turns out that the generalised eigenspaces $\frak A_{n,m}(\chi)$ in the corresponding spectral decomposition  
$$
\frak A_{n,m}=\oplus_{\chi} \frak A_{n,m}(\chi),
$$
where $\chi$ are certain homomorphisms $\chi: \frak{D}_{n,m} \rightarrow \mathbb C$, are in general not one-dimensional, similarly to the case of Jack--Laurent symmetric functions considered in our recent paper \cite{SV3}. We have shown there that the corresponding generalised eigenspaces have dimension $2^r$ and the image of the algebra of CMS integrals in the endomorphisms of such space is isomorphic to the tensor product of dual numbers  
$$
 \mathfrak A_r=\Bbb C[\varepsilon_1,\varepsilon_2,\dots,\varepsilon_r]/(\varepsilon_1^2,\,\varepsilon_2^2,\dots,\varepsilon_r^2).
 $$ 
It is known that the algebra $\frak A_r$ also appears as the algebra of the endomorphisms of the projective indecomposable modules over general linear supergroup (see Brundan-Stroppel \cite{Brund}). It is natural therefore to think about possible links between our generalised eigenspaces and projective modules.

The main result of this paper is a one-to-one correspondence between the finite-dimensional generalised 
eigenspaces of $\frak{D}_{n,m}$ and projective covers of certain irreducible finite-dimensional modules of $\mathfrak{gl}(n,2m)$. More precisely, we prove the following main theorem.

Let $Z(\frak g)$ be the centre of the universal enveloping algebra of $\frak g = \mathfrak{gl}(n,2m).$
For a $\frak g$-module $U$ we denote by $U^\frak k$ its part invariant under $\frak k=\frak{osp}(n,2m).$ 

\begin{thm} For any finite dimensional generalised eigenspace $\frak A_{n,m}(\chi)$ there exists  a unique projective indecomposable module $P$ over $\frak{gl}(n,2m)$ and a natural map 
$$
\Psi : (P^*)^{\frak k}\longrightarrow \frak A_{n,m}(\chi),
$$
which is an isomorphism of $Z(\frak g)$-modules. 

This establishes the bijection between the projective covers of spherically typical irreducible $\frak g$-modules and the finite dimensional generalised eigenspaces of the algebra of deformed CMS integrals $\frak{D}_{n,m}$ acting in $\frak A_{n,m}$.
\end{thm}

The corresponding projective modules can be described explicitly in terms of the highest weights of $\mathfrak{gl}(n,2m)$ under certain typicality conditions, which are natural generalisation of Kac's typicality conditions \cite{Kac}.

As a corollary we have an algorithm for calculating the composition quotients in Kac flag of the corresponding projective covers in the spherically typical case (which may have any degree of atypicality in the sense of \cite{Brun}). The number of the quotients is equal to the number of elements in the corresponding equivalence class, which can be described combinatorially, and equals $2^s$, where $s$ is the degree of atypicality (see sections 6 and 7 below). Our algorithm is  equivalent to Brundan-Stroppel  algorithm \cite{Brund} in this particular case, but our technique is different and uses the theory of the deformed CMS systems. 

The plan of the paper is following. In the next section we introduce the algebra $\frak{D}_{n,m}$ of quantum integrals of the deformed CMS system (mainly following \cite{SV4}) and study the corresponding spectral decomposition of its action on the algebra $\frak A_{n,m}.$

In section 3 we introduce symmetric Lie superalgebras and derive the formula for the radial  part of the corresponding Laplace-Beltrami operators. In particular, we show that for the four classical series of symmetric Lie superalgebras this radial part is conjugated to the deformed CMS operators introduced in \cite{SV}.

The rest of the paper is dealing mainly with the particular case corresponding to the symmetric pairs $(\frak g, \frak k)$ with $\frak g=\mathfrak{gl}(n,2m), \frak k=\mathfrak{osp}(n,2m).$ We call a finite dimensional $\frak g$-module $U$  spherical if the space of $\frak k$-invariant vectors $U^{\frak k}$ is non-zero. 
We describe the admissibility conditions on highest weight $\lambda,$ for which the corresponding Kac module $K(\lambda)$ is spherical.
Under certain assumptions of typicality we describe the conditions on admissible highest weights for irreducible modules to be spherical (see sections 5 and 6) and study the equivalence relation on the admissible weights defined by the equality of central characters.

These results are used  in section 7 to prove the main theorem, which implies in particular that any finite-dimensional generalised eigenspace contains at least one zonal spherical function, corresponding to an irreducible spherically typical $\frak g$-module.
In the last section we illustrate all this, including explicit formulas for the zonal spherical functions, in the simplest example of symmetric pair $X=(\mathfrak{gl}(1,2), \mathfrak{osp}(1,2)).$

\section{Algebra of deformed CMS integrals and spectral decomposition}

In this section we assume that the parameter $k$ is arbitrary nonzero, so everything is true for the special case $k=-\frac12$ as well.

To define the algebra of the corresponding CMS integrals $\mathfrak D_{n,m}$ it will be convenient to denote $x_{n+j}:=y_j,\,\, j=1,\dots,m$ and to introduce parity function 
$p(i)=0,\, i=1,\dots,n, \,\, p(i)=1,\,\, i=n+1,\dots, n+m.$
We also introduce the notation $$\partial_j=x_j\frac{\partial}{\partial x_j}, \,\,\,  j=1,\dots, n+m.$$

By definition the algebra  $\mathfrak D_{n,m}$ is generated by the deformed CMS integrals defined recursively in \cite{SV}. It will be convenient for us to use the following, slightly different choice of generators.

Define recursively the differential operators $\partial_{i}^{(p)}, 1\le i \le n+m, \, p \in \mathbb N$  as follows: for $p=1$ $$
\partial_{i}^{(1)}=k^{p(i)}\partial_i
$$
and for $p > 1$
\begin{equation}
\label{dif1}
\partial_{i}^{(p)}=\partial_{i}^{(1)}\partial_{i}^{(p-1)}-\sum_{j\ne i}k^{1-p(j)}\frac{x_{i}}{x_{i}-x_{j}}\left(\partial_{i}^{(p-1)}-\partial_{j}^{(p-1)}\right).
\end{equation}

Then the higher CMS integrals ${\mathcal{L}}_{p}$ are defined as the
sums 
\begin{equation}
\label{dif2} {\mathcal{L}}_{p}=\sum_{i\in I}k^{-p(i)}\partial_{i}^{(p)}.
\end{equation}
In particular, for $p=2$ we have
\begin{equation}
\label{L2}
\mathcal L_2=\sum_{i=1}^{n+m}k^{-p(i)} \partial_i^2-\sum_{i<j}^{n+m}\frac{x_{i}+x_{j}}{x_{i}-x_{j}}(k^{1-p(j)} \partial_i-k^{1-p(i)}\partial_j),
\end{equation}
which coincides with the deformed CMS operator (\ref{Lrad}).

\begin{thm}
 The operators ${\mathcal{L}}_{p}$ are quantum integrals of the deformed CMS system:
$$
[ {\mathcal{L}}_{p},  {\mathcal{L}}_{2}]=0.
$$
\end{thm}
\begin{proof}

Following the idea of our recent work \cite{SV4} introduce a version of quantum Moser $(n+m)\times (n+m)$-matrices  $L,\, M$ by
$$
L_{ii}=k^{p(i)}\partial_i-\sum_{j\ne i}k^{1-p(j)}\frac{x_{i}}{x_{i}-x_{j}},\,\,\, L_{ij}=k^{1-p(j)}\frac{x_i}{x_i-x_j},\, i\ne j
$$
$$
 M_{ii}=-\sum_{j\ne i}\frac{2k^{1-p(j)}x_ix_j}{(x_{i}-x_j)^2},\quad M_{ij}=\frac{2k^{1-p(j)}x_ix_j}{(x_{i}-x_j)^2},\, i\ne j.
$$
Note that matrix $M$ satisfies the relations 
$$
Me=e^*M=0,\,\,e=(\underbrace{1,\dots,1}_{n+m})^t,\,\, e^*=(\underbrace{1\dots,1}_{n},\underbrace{1/k,\dots,1/k}_{m}).
$$
Define also matrix Hamiltonian $H$ by
$$
H_{ii}=\mathcal L_2,\,\,\,H_{ij}=0,\,i\ne j.
$$
Then it is easy to check that these matrices satisfy Lax relation
$$
[L,H]=[L,M].
$$
Indeed, matrix $L$ is different from the Moser matrix (43) from \cite{SV4}
by rank one matrix $e\otimes e^*$, which does not affect the commutator $[L,M]$ because of the relations $Me=e^*M=0.$
This implies as in \cite{SV4,UHW} that the "deformed total trace"
$$
{\mathcal{L}}_{p}=\sum_{i,j}k^{-p(i)}(L^p)_{ij}
$$
commute with ${\mathcal{L}}_{2}.$ 
\end{proof}

Define now the {\it Harish-Chandra homomorphism}
$$
\varphi: \mathfrak D_{n,m} \rightarrow \mathbb C[\xi_1, \dots, \xi_{n+m}]
$$
by the conditions (cf. \cite{SV}):
$$\varphi (\partial_i)=\xi_i, \quad \varphi \left(\frac{x_{i}}{x_{i}-x_{j}}\right)=1, \, \,\, {\text {if}} \,\, i<j. 
$$
In particular, $d_i^{(p)}(\xi):=\varphi(\partial_i^{(p)})$  satisfy the following recurrence relations
\begin{equation}
\label{recu}
d_i^{(p)}=d_i^{(1)}d_{i}^{(p-1)}-\sum_{j> i}k^{1-p(j)}(d_i^{(p-1)}-d_j^{(p-1)}),
\end{equation}
which determine them uniquely with $d_i^{(1)}=k^{p(i)}\xi_i, \, i=1,\dots, n+m.$

Let $\rho(k) \in \mathbb C^{n+m}$ be the following deformed analogue of the Weyl vector 
\begin{equation}
\label{rhok}
\rho(k)=\frac12\sum_{i=1}^n(k(2i-n-1)-m)e_i+\frac12\sum_{j=1}^{m}(k^{-1}(2j-m-1)+n)e_{j+n}
\end{equation}
and consider the bilinear form $( , )$ on $\mathbb C^{n+m}$ defined in the basis $e_1,\dots, e_{n+m}$ by
$$
(e_i, e_i)=1, \,\, i=1,\dots, n, \quad (e_j, e_j)=k, \,\, j=n+1,\dots, n+m.
$$

\begin{thm}\label{hc} \cite{SV}
 Harish-Chandra homomorphism is injective and its image is the subalgebra
$\Lambda_{n,m}(k) \subset \mathbb C[\xi_1, \dots, \xi_{n+m}]$
consisting of polynomials with the following properties:
$$
f(w(\xi+\rho(k)))=f(\xi+\rho(k)), \quad w \in S_n\times S_m
$$
and 
for every $i \in \{1,\dots, n\},\,\,j \in \{n+1,\dots, n+m\}$ 
\begin{equation}
\label{hc1}
f(\xi-e_i+e_j)=f(\xi)
\end{equation}
on the hyperplane $(\xi+\rho(k), e_i-e_j)=\frac12(1+k).$
\end{thm}
\begin{corollary} Operators $\mathcal L_p$ commute with each other.
\end{corollary}

From the results of \cite{SV5} it follows that $\mathcal L_p$ generate the same algebra $\mathfrak D_{n,m}$ as commuting CMS integrals from \cite{SV}, which gives another proof of their commutativity. 

Let now $\frak A_{n,m}$ be the algebra consisting of  $S_n\times S_m$-invariant Laurent polynomials
$f\in\Bbb C[x_1^{\pm1},\dots, x_n^{\pm1},y_1^{\pm1}\dots,y_m^{\pm1}]^{S_n\times S_m}$
satisfying the quasi-invariance condition
\begin{equation}
\label{quasix}
x_i\frac{\partial f}{\partial x_i}-ky_j\frac{\partial f}{\partial y_j}\equiv 0
\end{equation}
on the hyperplane $x_i=y_j$ for all $i=1,\dots,n$, $j=1,\dots,m$ with $k$ being arbitrary and the same as in the definition of $\mathfrak D_{n,m}$. We claim that the algebra $\mathfrak D_{n,m}$ preserves it.

For any Laurent polynomial 
$$
f=\sum_{\mu \in X_{n,m}}c_{\mu}x^{\mu}, \quad X_{n,m}=\mathbb Z^n\oplus\mathbb Z^m
$$
consider the set $M(f)$ consisting of $\mu$ such that $c_{\mu}\ne0$ and define the {\it support} $S(f)$ as the intersection of the convex hull of $M(f)$ with $X_{n,m}.$

\begin{thm}\label{in} The operators  ${\mathcal{L}}_{p}$ for all
$p=1,2,\dots$ map the algebra  $\frak A_{n,m}$ to itself and preserve the support: for any $ D\in{\mathfrak D}_{n,m}$  and $f\in \frak{A}_{n,m}$
$$
S(Df)\subseteq S(f).
$$
\end{thm}

\begin{proof} The first part follows from the fact that if $f\in \frak A_{n,m}$ then $\partial_{i}^{(p)}f$ is a polynomial. The proof is essentially repeating the arguments from \cite{SV1} (see theorem 5 and lemmas 5 and 6), so we will omit it.

To prove the second part it is enough to show that $S(\partial_{i}^{(p)} f)\subseteq S(f)$ for any  
$f\in \frak A_{n,m}.$ From the recursion (\ref{dif1}) we see that it is enough to prove that
$$
g=\frac{x_{i}}{x_{i}-x_{j}}\left(\partial_{i}^{(p-1)}-\partial_{j}^{(p-1)}\right)(f)
$$
is a polynomial and $S(g)\subseteq S(f).$ Denote  $\left(\partial_{i}^{(p-1)}-\partial_{j}^{(p-1)}\right)(f)$ as $h(x),$ which is known to be a Laurent polynomial. By induction assumption $S(h) \subseteq S(f)$. Since 
$$
h(x)=\left(1-\frac{x_j}{x_i}\right)g(x)$$
and the support of a product of two Laurent polynomials is the Minkowski sum of the supports of the factors  this implies that $S(g)\subseteq S(h) \subseteq S(f)$.
\end{proof}

Now we are going to investigate the spectral decomposition of the action of the algebra of CMS integrals $\mathfrak D_{n,m}$ on $\frak{A}_{n,m}$. 

We will need the following {\it partial order} on the set of integral weights $\lambda \in X_{n,m} = \mathbb Z^{n+m}$: we say that
$\mu \preceq \lambda$ if and only if
\begin{equation}
\label{partial}
\mu_1 \le \lambda_1, \, \mu_1+\mu_2 \le \lambda_1+\lambda_2, \dots, \mu_1+\dots +\mu_{n+m} \le \lambda_1+\dots +\lambda_{n+m}.
\end{equation}

\begin{proposition}\label{sup} 
Let $f \in \mathfrak A_{n,m}$ and $\lambda$ be a maximal element of $M(f)$ with respect to partial order. Then for any $D \in \mathfrak D_{n,m}$ there is no $\mu$ from $M(D(f)), \, \mu \ne \lambda$ such that $\lambda \preceq \mu.$ The coefficient at $x^{\lambda}$ in $D(f)$ is $\varphi(D)(\lambda)c_\lambda$, where $c_\lambda$ is the coefficient at $x^{\lambda}$  in $f.$

If $\lambda$ is the only maximal element of $M(f)$ then $\mu \preceq \lambda$ for any  $\mu$ from $M(D(f)).$
\end{proposition} 

\begin{proof}  It is enough to prove this only for $D=\partial_{i}^{(p)}$. We will do it by induction on $p$. 
In the notations of the proof of Theorem \ref{in} let us assume that there is $\mu \in M(g)$ such that
$\lambda \preceq \mu, \, \mu \ne \lambda.$ Without loss of generality we can assume that $\mu$ is maximal in $M(g).$

From $h(x)=(1-x_j/x_i)g(x)$
with $i<j$ it follows that $\mu$ is also maximal in $M(h)$, which contradicts the inductive assumption. This implies that the coefficient at $x^\lambda$ in $\partial_{i}^{(p)}(f)$ satisfies
the same recurrence relations (\ref{recu}) with the initial conditions multiplied by $c_\lambda.$
This proves the first part. The proof of the second part is similar.
\end{proof}

Let $\chi: \mathfrak D_{n,m} \rightarrow \mathbb C$ be a homomorphism and define the corresponding {\it generalised eigenspace} $\frak A_{n,m}(\chi)$ as the set of all $f\in \frak A_{n,m}$ such that for every $D\in \mathfrak D_{n,m}$ there exists $N\in \mathbb N$ such that $(D-\chi(D))^N(f)=0.$
 If the dimension of $\frak A_{n,m}(\chi)$ is finite then such $N$ can be chosen independent on $f.$ 
 
\begin{proposition}\label{dec} Algebra $\frak A_{n,m}$  as a module  over the algebra $\mathfrak D_{n,m}$ can be decomposed in a direct sum of generalised eigenspaces
\begin{equation}
\label{sum}
\frak A_{n,m}=\oplus_{\chi}\frak A_{n,m}(\chi),
\end{equation}
where the sum is taken over the set of some homomorphisms $\chi$ (explicitly described below).
\end{proposition}

 \begin{proof} 
 Let $f\in \frak A_{n,m}$ and define a vector space
$$
V(f)=\{g\in \frak A_{n,m}\mid S(g)\subseteq S(f)\}.
$$
By Theorem \ref{in} $V(f)$ is a finite dimensional module over $\mathfrak D_{n,m}.$ 
Since the proposition is true for every finite-dimensional modules the claim now follows.
\end{proof}

Now we describe all homomorphisms $\chi$ such that $\frak A_{n,m}(\chi)\ne0$.
We say that the integral weight $\lambda \in X_{n,m}\in\mathbb Z^{n+m}$ {\it dominant} if 
$$
\lambda_1 \ge \lambda_2 \ge \dots \ge \lambda_n, \quad \lambda_{n+1} \ge \lambda_{n+2} \ge \dots \ge \lambda_{n+m}.
$$
The set of dominant weights is denoted $X_{n,m}^+.$

For every $\lambda\in X_{n,m}^{+}$  we define the homomorphism $\chi_\lambda: \mathfrak D_{n,m} \rightarrow \mathbb C$ by
$$
\chi_{\lambda}(D)=\varphi(D)(\lambda), \,\, D\in \mathfrak D_{n,m}
$$
where $\varphi$ is the Harish-Chandra homomorphism.

\begin{proposition}\label{max}
$1)$ For any $\lambda\in X_{n,m}^+$ there exists $\chi$ and $f\in \frak A_{n,m}(\chi)$, which has the only maximal term $x^{\lambda}$.

$2)$  $\frak A_{n,m}(\chi)\ne0$ if and only if there exists $\lambda\in X_{n,m}^+$ such that $\chi=\chi_{\lambda}$.

$3)$ If  $\frak A_{n,m}(\chi)$  is finite dimensional then its dimension is equal to the number of $\lambda \in X_{n,m}^+$ such that $\chi_{\lambda}=\chi$.
\end{proposition}

\begin{proof}   
Let $\mu_1=\lambda_1, \dots, \mu_n=\lambda_n, \,\, \, \nu_1=\lambda_{n+1}, \dots, \nu_{m}=\lambda_{n+m}$. Consider the Laurent polynomial 
$$
g(x,y)=s_{\mu}(x)s_{\nu}(y)\prod_{i,j}(1-y_j/x_i)^2
$$
where $s_{\mu}(x), \, s_{\nu}(y)$ are the Schur polynomials \cite{Ma}. It is easy to check that $g$ belongs to the algebra $\frak A_{n,m}$ and has the only maximal weight $\lambda$.  By Proposition \ref{dec} we can write $g=g_1+\dots+g_N$, where  $g_i$ belong to different generalised eigenspaces. Therefore there exists $i$ such that $\lambda \in M(g_i)$. Since $g_i$ can be obtained from $g$ by some element from the algebra $\mathfrak A_{n,m}$ (which is a projector to the corresponding generalised eigenspace in some finite-dimensional subspace containing $g$), then $\lambda$ is  the only maximal element of $M(g_i)$ by Proposition \ref{sup}. This proves the first part.

 Let $ \frak A_{n,m}(\chi)\ne0$. Pick up a nonzero element $f$ from this subspace  and choose some maximal element $\lambda^{(1)}$  from $M(f)$ and an operator $D \in \mathfrak D_{n,m}$. Then according to Proposition \ref{sup} element $x^{\lambda^{(1)}}$ does not enter in  $f_1=(D-\chi_{\lambda^{(1)}}(D))(f)$ and $S(f_1) \subset S(f)$. Repeating this procedure we get the sequence  of nonzero elements $f_0=f,\,f_1,\,\dots,f_N$ and the numbers $a_1=\chi_{\lambda^{(1)}}(D),\dots,a_N=\chi_{\lambda^{(N)}}(D)$ such that
$$
f_i=(D-a_i)f_{i-1},\,i=1,\dots,N,\quad (D-a_N)f_{N-1}=0.
$$
Therefore 
$$
P(t)=\prod_{i=1}^N(t-a_i)
$$
is a minimal polynomial for $D$ in the subspace $<f_0,\dots,f_{N-1}>$. But this subspace is in $\frak A_{n,m}(\chi)$. Therefore this polynomial should be some power of $t-\chi(D)$ and hence $a_1=a_2=\dots=a_N=\chi(D)$. In particular, this implies that $\chi(D)=a_1=\chi_{\lambda^{(1)}}(D)$ for some $\lambda^{(1)}\in X^+_{n,r}$ as required.

Conversely, let $\lambda\in X_{n,m}^+$. According to the  first part there exists $\chi$ and $f\in \frak A_{n,m}(\chi)$ such that $\lambda$ is its maximal weight. Therefore the previous considerations show that $\chi=\chi_{\lambda}$ and thus $ \frak A_{n,m}(\chi_\lambda)\ne0.$ 

To prove the third  part suppose that $\frak A_{n,m}(\chi)$ is finite dimensional and that $\lambda^{(1)},\dots,\lambda^{(N)}$ are all different elements from $X_{n,m}^+$ such that $\chi_{\lambda^{(i)}}=\chi,\, i=1,\dots, N$. According to the first two parts there exists $f_i \in \frak A_{n,m}(\chi)$  with the only maximal weight $\lambda^{(i)}$. It is easy to see  that  $f_1,\dots,f_N$ are linearly independent.  
To show that they form a basis consider any $f \in \frak A_{n,m}(\chi)$ and take a maximal weight $\mu$ from $M(f).$ According to Proposition \ref{sup} $\chi_\mu=\chi$ and thus $\mu$ must coincide with one of $\lambda^{(i)}.$ By subtracting from $f$ a suitable multiple of $f_i$ and using induction we get the result.
 \end{proof}

\begin{corollary}
The set of homomorphisms in Proposition \ref{dec} consists of $\chi=\chi_\lambda, \,\, \lambda \in X^+_{n,m}.$
\end{corollary}

\section{Symmetric Lie superalgebras and Laplace-Beltrami operators}

We will be using an algebraic approach to the theory  of symmetric superspaces based on the notion of symmetric Lie superalgebras going back to Dixmier \cite{Dix}. More geometric approach with relation to physics and random-matrix theory can be found in Zirnbauer \cite{Zirn}. For the classification of real simple Lie superalgebras and symmetric superspaces see Serganova \cite{Se}. 

 {\it Symmetric Lie superalgebra} is a pair $(\frak g,\theta),$ where $\frak g$ is a complex Lie superalgebra, which will be assumed to be basic classical \cite{Kac0},\footnote{Strictly speaking, the Lie superalgebra $\frak {gl}(n,m)$ is not basic classical, but it is more convenient for us to consider than $\frak {sl}(n,m).$}  and  $\theta$ is an involutive automorphism of $\frak g$. We have the decomposition 
$ \mathfrak g = \mathfrak k \oplus \frak p,$
where $\mathfrak k$ and $\frak p$ are $+1$ and $-1$ eigenspaces of $\theta:$
$$
[\frak k, \frak k]\subset \frak k, \, [\frak k, \frak p]\subset \frak p, \, [\frak p, \frak p]\subset \frak k.
$$
Alternatively, one can talk about {\it symmetric pair} $X=(\frak g, \frak k).$

In this paper we restrict ourselves by the following 4 classical series of symmetric pairs (in Cartan's notations \cite{Hel0, Zirn}):
$$AI/AII=(\frak {gl}(n,2m), \frak {osp}(n,2m)),\,\,\, DIII/CI=(\frak{osp} (2l, 2m), \frak{gl}(l,m)),$$
\begin{equation}
\label{4series}
AIII=(\frak{gl} (n_1+n_2, m_1+m_2), \frak{gl}(n_1,m_1)\oplus\frak{gl}(n_2,m_2)),
\end{equation}
$$BDI/CII=(\frak{osp} (n_1+n_2, 2m_1+2m_2), \frak{osp}(n_1,2m_1)\oplus\frak{osp}(n_2,2m_2)).$$ 
In fact, we will give all the details only for the first series, which will be our main case (see next Section).
For a more general approach we refer to the work by Alldridge et al \cite{Ald}.

Commutative subalgebra  $\frak a \subset\frak p$ is called {\it Cartan subspace} if it is reductive in $\frak g$ and the centraliser  of $\frak a$ in $\frak p$ coincides with $\frak a$ \cite{Dix}. We will consider only the cases when Cartan subspace can be chosen to be even ("even type" in the terminology of \cite{Ald}). 

The Lie superalgebra $\frak g$ has an even invariant supersymmetric bilinear form with restriction on $\frak a$ being non-degenerate. The corresponding quadratic form on $\frak a$ we denote $Q.$

We have the decomposition of $\frak g$ with respect to $\frak a$ into nonzero eigenspaces 
$$
\frak g=\frak g^\frak a_0\oplus \bigoplus_{\alpha\in R(X)}\frak g^\frak a_{\alpha}.
$$
The corresponding set $R(X) \subset \frak a^*$ is called {\it restricted root system} of $X$ and $\mu_{\alpha}=sdim \,  \frak g^\frak a_{\alpha}$ are called {\it multiplicities}, where $sdim \,  \frak g^\frak a_{\alpha}$ is the super dimension:
$sdim \, g^\frak a_{\alpha}=\dim g^\frak a_{\alpha}$ for even roots and $sdim \, g^\frak a_{\alpha}=-\dim g^\frak a_{\alpha}$ for odd roots.

For the symmetric pairs $X=(\frak {gl}(n,2m), \frak {osp}(n,2m))$ of type $AI/AII$ we have  the following root system
consisting of the even roots $\pm (x_i-x_j), \, 1\le i < j \le n$ with multiplicity $\mu=1$, $\pm (y_i-y_j), \, 1\le i <  j \le m$ with multiplicity $\mu=4$ and odd roots 
$\pm (x_i-y_j), \, 1\le i\le n, \, 1\le j\le m$ with multiplicity $\mu=-2$. The corresponding invariant quadratic form is
\begin{equation}
\label{Q}
Q= x_1^2+\dots+x_n^2+k^{-1}(y_1^2+\dots+y_m^2)
\end{equation}
with $k=-1/2$ (see the next section).

For the remaining 3 classical series we have the following restricted root systems of $BC(n,m)$ type, see \cite{OP, Ald}.

For $X=(\frak{osp} (2l, 2m), \frak{gl}(l,m))$ of type $DIII/CI$ the restricted root system depends on the parity of $l.$
For odd $l=2n+1$ the restricted even roots are $\pm x_i$ with $\mu=4$, $\pm 2x_i$ with $\mu=1$ for $i=1,\dots, n,$
$\pm x_i \pm x_j$ with $\mu=4$ for $1 \le i < j \le n,$ $\pm 2y_i$ with $\mu=1$ for $i=1,\dots, m,$
$\pm y_i \pm y_j$ with $\mu=1$ for $1 \le i < j \le m$ and odd roots $\pm x_i \pm y_j$  
and $\pm y_j$ with $\mu=-2$ with for $1 \le i \le n,\, 1 \le j \le m.$ The quadratic form $Q$ by (\ref{Q}) with $k=-2.$

For even $l=2n$ the restricted even roots are $\pm 2x_i$ with $\mu=1$ for $i=1,\dots, n,$
$\pm x_i \pm x_j$ with $\mu= 4$ for $1 \le i < j \le n,$ $\pm 2y_i$ with $\mu=1$ for $i=1,\dots, m,$
$\pm y_i \pm y_j$ with $\mu=1$ for $1 \le i < j \le m$ and odd roots $\pm x_i \pm y_j$  
 with $\mu=-2$ with for $1 \le i \le n,\, 1 \le j \le m.$ The quadratic form $Q$ is given by (\ref{Q}) with $k=-2.$

For the symmetric pairs $(\frak{gl} (n_1+n_2, m_1+m_2), \frak{gl}(n_1,m_1)\oplus\frak{gl}(n_2,m_2))$ of type $AIII$ 
the even type means that $(n_1-m_1)(n_2-m_2)\geq 0$ (see \cite{Ald}). We have then
$n=\min(n_1,n_2), \, m=\min(m_1,m_2)$ and the even roots
$\pm x_i$ with $\mu=2|n_1-n_2|$, $\pm 2x_i$ with $\mu=1$ for $i=1,\dots, n,$
$\pm x_i \pm x_j$ with $\mu=2$ for $1 \le i < j \le n,$ $\pm y_i$ with $\mu=2|m_1-m_2|,$ $\pm 2y_i$ with $\mu=1$ for $i=1,\dots, m,$
$\pm y_i \pm y_j$ with $\mu=2$ for $1 \le i < j \le m$ and odd roots $\pm x_i \pm y_j$  
with $\mu=-2$, $\pm x_i$ with $\mu=-2|m_1-m_2|$, $\pm y_j$ with $\mu=-2|n_1-n_2|$ for $1 \le i \le n,\, 1 \le j \le m.$ The  form $Q$ is given by (\ref{Q}) with $k=-1.$

For the even type $BDI/CII$ pairs $(\frak{osp} (n_1+n_2, 2m_1+2m_2), \frak{osp}(n_1,2m_1)\oplus\frak{osp}(n_2,2m_2))$ with $(n_1-m_1)(n_2-m_2)\geq 0$ we have again
$n=\min(n_1,n_2), \,\, m=\min(m_1,m_2)$ and the even roots
$\pm x_i$ with $\mu=|n_1-n_2|$ for $i=1,\dots, n,$
$\pm x_i \pm x_j$ with $\mu=1$ for $1 \le i < j \le n,$ $\pm y_i$ with $\mu=4|m_1-m_2|,$ $\pm 2y_i$ with $\mu=3$ for $i=1,\dots, m,$
$\pm y_i \pm y_j$ with $\mu=4$ for $1 \le i < j \le m$ and odd roots $\pm x_i \pm y_j$  
with $\mu=-2$, $\pm x_i$ with $\mu=-2|m_1-m_2|$, $\pm y_j$ with $\mu=-2|n_1-n_2|$ for $1 \le i \le n,\, 1 \le j \le m.$ The form $Q$ is given by (\ref{Q}) with $k=-1/2.$



Let $U(\frak g)$ be the universal enveloping algebra of $\frak g.$ Let $e_1, \dots, e_N$ be a basis in $\frak g.$ The dual space $U(\frak g)^*$ is known to be the algebra isomorphic to the algebra of formal series $\mathbb C[[X_1, \dots, X_N]],$ where $X_1, \dots, X_N \in \frak g^*$ is a dual basis (see Dixmier \cite{Dix}, Chapter 2).

By a {\it zonal function} for the symmetric pair $X=(\mathfrak{g}, \mathfrak k)$ we mean a linear functional $f \in U(\frak g)^*$, which is two-sided $\mathfrak{k}$-invariant:
$$
f(xu)=f(ux)=0,\,\,x\in\frak k,\,\, u\in U(\frak{g}).
$$
The space of such functions we denote $\mathcal Z(X) \subset U(\frak g)^*.$

Let $Y=\frak{k}U(\frak{g})+ U(\frak{g})\frak{k}$ be a subspace in $U(\frak{g}),$ on which the zonal functions vanish. Let also $U(\frak a)=S(\frak a)$ be the symmetric algebra of $\frak a.$

\begin{proposition}\label{ress1} 
$$
U(\frak{g})=S(\frak{a})+Y.
$$
\end{proposition}

\begin{proof} 

For any $x\in\frak g$ set 
$$
x^+=\frac12(x+\theta(x)), \quad x^-=\frac12(x-\theta(x)).
$$
Let $\alpha \in \frak h^*$ be a root of $\mathfrak g$ and $T_{\alpha}: U(\frak a) \to U(\frak a)$ be the automorphism defined by
$x \to x + \alpha(x), \, x \in \frak a.$ Define $R_{\alpha}^{\pm}: U(\frak a) \to U(\frak a)$ by
$$R_{\alpha}^{+}=\frac12(T_{\alpha}+T_{-\alpha}), \quad R_{\alpha}^{-}=\frac12(T_{\alpha}-T_{-\alpha}).
$$
Let $X_{\alpha}\in \frak g_{\alpha}, X_{-\alpha}\in\frak g_{-\alpha}$ be corresponding root vectors and define  $$h_{\alpha}=[X_{\alpha},X_{-\alpha}] / (X_{\alpha}, X_{-\alpha}) \in \frak h.$$

We will need the following lemma, which can be checked directly.

\begin{lemma}\label{comm} For any $u \in U(\frak a)$ the following equalities hold true:

{\em i)} $X_{\alpha}^+u=R_{\alpha}^+uX_{\alpha}^+-R_{\alpha}^-uX_{\alpha}^-$

{\em ii)} $R_{\alpha}^- uX_{\alpha}X_{-\alpha}-T_{\alpha} u
[X_{\alpha}^+, X_{-\alpha}^-]\in Y$

{\em iii)} $[X_{\alpha}^+,X_{\alpha}^-]=\frac12h_{\alpha}^-(X_{\alpha}, X_{-\alpha}). $

\end{lemma}

 To prove the proposition it is enough to show that for  $q>0$
$$
v=uX^-_{\alpha_1} \ldots X^-_{\alpha_q}\in Y
$$
for any roots ${\alpha_1},\ldots, \alpha_q$, where  $u\in S(\frak{a})$. We prove this by induction in $q$. 

If $q=1$ and  $w\in S(\frak{a})$, then by the first part of Lemma \ref{comm}
we have $R_{\alpha}^-wX_{\alpha}^-=R_{\alpha}^+wX_{\alpha}^+-X_{\alpha}^+w$, which clearly belongs to $Y.$
But any  $u\in S(\frak a)$ can be represented in the form   $u=R_{\alpha}^-w$ for some $ w\in S(\frak a)$, therefore $uX_{\alpha}^-\in Y$.
Let now $q>1$. Then modulo $Y$ we have using Lemma \ref{comm}
$$
R^-_{{\alpha_1} }uX^-_{\alpha_1}
 \ldots X^-_{\alpha_q} = R^+_{\alpha_1} uX^+_{\alpha_1} X^-_{\alpha_2}\ldots X^-_{\alpha_q}-X^+_{\alpha_1}uX^-_{\alpha_2}\ldots X^-_{\alpha_q}
$$
$$
{\equiv}R^+_{\alpha_1} uX^+_{\alpha_1} X^-_{\alpha_2}
\ldots X^-_{\alpha_q}
\equiv R^+_{\alpha_1} u\cdot [X^+_{\alpha_1},
X^-_{\alpha_2} \ldots X^-_{\alpha_q}] 
$$
$$ = R^+_{\alpha_1}
u\cdot[X^+_{\alpha_1}, X^-_{\alpha_2}] X^-_{\alpha_3}\ldots X^-_{\alpha_q}
 +R^+_{\alpha_1} u X^-_{\alpha_2}[X^+_{\alpha_1}, X^-_{\alpha_3}]\ldots
X^-_{\alpha_q}+\ldots\in Y
$$
by inductive assumption. 
 \end{proof}
 
 Let  $\alpha\in R$  be a root of $\frak g$ such that the restriction of   $\alpha$ on $\frak{a}$ is not zero. Let also $f \in \mathcal Z(X)$ be a two sided $\frak k$-invariant functional on $U(\frak{g})$. By proposition \ref{ress1} $f$ is uniquely determined by its restriction to $U(\frak a) = S(\frak a),$ and thus we can consider $\mathcal Z(X)$ as a subalgebra $S(\frak a)^*.$ 

Identify $S(\frak a)^*$ with the algebra of formal power series as follows (see \cite{Dix}). Let $e_1, \dots, e_N$ be a basis in $\frak a$ and $x_1, \dots, x_N \in \frak a^*$ be the dual basis. Then we can define for any $f \in S(\frak a)^*$ the formal power series $\hat f \in \mathbb C[[x_1,\dots,x_N]]$ by 
$$
\hat f = \sum_{M \in \mathbb Z_+^N} f(e_M)x^M,
$$
where 
$$e_M=\frac{1}{m_1!\dots m_N!}e_1^{m_1}\dots e_N^{m_N} \in U(\frak a), \,\, x^M=x_1^{m_1}\dots x_N^{m_N}.$$
It is easy to see that the operator of multiplication by, say, $e_1$ corresponds to the partial derivative $\frac{\partial}{\partial x_1}$ in this realisation:
$$
\hat f(u e_1)=\frac{\partial}{\partial x_1} \hat f(u).
$$ 
 Similarly, the shift operator 
$T_\lambda, \, \lambda \in \frak a^*$ corresponds to multiplication by $e^{\lambda}$:
$$\hat f(T_\lambda u)=e^\lambda \hat f(u), \,\, u \in S(\frak a).$$

Let $S \subset S(\mathfrak a)^*$ be the multiplicative set generated by  $e^{2\alpha}-1, \, \alpha \in R(X)$ and $S(\mathfrak a)^*_{loc}=S^{-1}S(\mathfrak a)^*$ be the corresponding localisation.

Let now $\frak h$ be a Cartan subalgebra of $\frak g$, $R$ be the root system of $\frak g$ and $X_\alpha$ are the corresponding root vectors with respect to $\frak h.$ 

Choose an orthogonal basis $h_i \in \mathfrak h, \, i=1,\dots, r$ and define the {\it quadratic Casimir element} $\mathcal C_2$ from the centre  $Z(\frak g)$ of the universal enveloping algebra $U(\frak g)$ by
\begin{equation}
\label{casimir}
\mathcal C_2= \sum_{i=1}^r \frac{h_i^2}{(h_i,h_i)} + \sum_{\alpha \in R} \frac{X_\alpha X_{-\alpha}}{(X_{-\alpha}, X_{\alpha})},
\end{equation}
where the brackets denote the invariant bilinear form on $\frak g.$ 

It can be defined invariantly as an image of the element of $\frak g \otimes \frak g$ representing the invariant form itself and determines the corresponding {\it Laplace-Beltrami operator} $\frak L$  on $X$ acting on left $\frak k$-invariant functions $f \in \frak F(X)=U(\frak g)^{*\frak k}$ (which are algebraic analogues of the functions on the symmetric superspace $X=G/K$) by
$$\frak L f(x)=f(x \mathcal C_2), \, x \in U(\frak g).$$ 

The restriction of the invariant bilinear form on $\frak g$  to $\frak a$ is a non-degenerate form, which we also denote by $( , ).$ Let $\Delta$ be the corresponding Laplace operator on $\frak a$  and 
$\partial_\alpha, \, \alpha \in \frak a^*$ be the differential operator on $\frak a$ defined by 
\begin{equation}
\label{defini}
\partial_\alpha e^{\lambda}=(\alpha, \lambda) e^{\lambda}.
\end{equation}

Consider the following operator $\frak L_{rad}: S(\mathfrak a)^* \to S(\mathfrak a)^*$ defined by
\begin{equation}
\label{radpart}
\frak L_{rad}=\Delta+\sum_{\alpha \in R_+(X)}\mu_\alpha \frac{e^{2\alpha}+1}{e^{2\alpha}-1} \,\, \partial_\alpha,
\end{equation}
where the sum is taken over positive restricted roots considered with multiplicities $\mu_\alpha$. 
This operator is the {\it radial part} of the Laplace-Beltrami operator $\frak L$ in the following sense.

\begin{proposition}
 The following diagram is commutative
\begin{equation}
\label{commutdun1}
\begin{array}{ccc}
\mathcal Z(X)&\stackrel{\frak L}{\longrightarrow}&\mathcal Z(X)\\ \downarrow
\lefteqn{i^*}& &\downarrow \lefteqn{i^*}\\
S(\mathfrak a)^*_{loc}&\stackrel{\frak L_{rad}}{\longrightarrow}& 
S(\mathfrak a)^*_{loc}.\\
\end{array}
\end{equation}
\end{proposition}

\begin{proof}
For any root $\alpha$ of $\frak g$ define the operators 
 $D_{\alpha}, \, \partial_{\alpha}:  \mathcal Z(X) \to S(\frak a)^*$ by 
 $$D_{\alpha}(f)(u)=f\left(\frac{uX_{\alpha}X_{-\alpha}}{(X_{-\alpha}, X_{\alpha})}\right), \,\, 
\partial_{\alpha}(f)(u)= f(u h_{\alpha}^-),\,\, u\in S(\frak{a}),$$ 
where we consider $\mathcal Z(X)$ as a subset of $S(\frak a)^*.$ One can check that the definition of the operator $\partial_\alpha$ agrees with (\ref{defini}).

We claim that the operators $D_{\alpha}, \, \partial_{\alpha}$ in the formal power series realisation satisfy the relation 
\begin{equation}
\label{Dalpha}
(e^{\alpha}-e^{-\alpha})D_{\alpha}=(-1)^{p(\alpha)}e^{\alpha}\partial_{\alpha},
\end{equation}
where $p(\alpha)$ is parity function: $p(\alpha)=0$ for even roots and $p(\alpha)=1$ for odd roots.

Indeed, since the restriction of $f \in \frak F(X)$ on $Y$ vanishes, from parts $ii)$  and $iii)$ of  Lemma \ref{comm} it follows that
$$
\hat f\left(\frac{R_{\alpha}^- uX_{\alpha}X_{-\alpha}}{(X_{-\alpha}, X_{\alpha})}\right)=\hat f\left(\frac{T_{\alpha} u
[X_{\alpha}^+, X_{-\alpha}^-]}{(X_{-\alpha}, X_{\alpha}) }\right)  = \frac{(-1)^{p(\alpha)}}{2}e^\alpha \hat f(u
h_{\alpha}^-),
$$
since $(X_{-\alpha},X_{\alpha})=(-1)^{p(\alpha)}(X_{\alpha},X_{-\alpha}).$
Since $R_{\alpha}^- =\frac 12(T_{\alpha}-T_{-\alpha})$ we have
$\hat f(R_{\alpha}^- uX_{\alpha}X_{-\alpha})=\frac 12(e^{\alpha}-e^{-\alpha})\hat f(uX_{\alpha}X_{-\alpha})$ and thus the claim.

In the localisation $S(\mathfrak a)^*_{loc}$ we can write the operator $D_\alpha$ as
\begin{equation}
\label{dalpha}
D_\alpha=\frac{(-1)^{p(\alpha)}e^\alpha}{e^\alpha-e^{-\alpha}}\partial_\alpha=\frac{(-1)^{p(\alpha)}e^{2\alpha}}{e^{2\alpha}-1}\partial_\alpha
\end{equation}
and extend it to the whole $S(\mathfrak a)^*_{loc}$.

Summing over all $\alpha \in R$ and taking into account that the multiplicities $\mu_\alpha$ are defined with the sign $(-1)^{p(\alpha)}$ after the restriction to $\frak a$ we have the second term in formula (\ref{radpart}). One can check that the first part of the Casimir operator (\ref{casimir}) gives the Laplace operator $\Delta.$
\end{proof}

\begin{corollary}
For 4 classical series of symmetric pairs (\ref{4series}) of even type the radial parts of Laplace-Beltrami operators are conjugated to the deformed CMS operators of classical type.
\end{corollary}

More precisely, for the classical series $X=(\frak{gl} (n, 2m), \frak{osp} (n,2m))$ the corresponding radial part (\ref{radpart}) is conjugated to the deformed CMS operator related to generalised root system of type $A(n-1,m-1)$ from \cite{SV} with parameter $k=-1/2$ (as it was already pointed out in \cite{Ser}).

For three other classical series the corresponding radial part is conjugated to the following deformed CMS operator of type $BC(n,m)$ introduced in \cite{SV}
\begin{eqnarray}
\label{bcnm} L& =& -\Delta_n
-k \Delta_m +\sum_{i<j}^{n}\left(\frac{2k(k+1)}{\sinh^2(x_{i}-x_{j})}+\frac{2k(k+1)}{\sinh^2(x_{i}+x_{j})}\right)\nonumber \\& &
+\sum_{i<j}^{m}\left(\frac{2(k^{-1}+1)}{\sinh^2(y_{i}-y_{j})}+\frac{2(k^{-1}+1)}{\sinh^2(y_{i}+y_{j})}\right)
\nonumber
\\& & +\sum_{i=1}^{n}\sum_{j=1}^{m}\left(\frac{2(k+1)}{\sin^2(x_{i}-y_{j})}+
\frac{2(k+1)}{\sinh^2(x_{i}+y_{j})}\right) +\sum_{i=1}^n
\frac{p(p+2q+1)}{\sinh^2x_{i}} \nonumber \\& & +\sum_{i=1}^n
\frac{4q(q+1)}{\sinh^22x_{i}} +\sum_{j=1}^m \frac{k
r(r+2s+1)}{\sinh^2y_{j}}+\sum_{j=1}^m \frac{4k
s(s+1)}{\sinh^22y_{j}},
\end{eqnarray}
where the parameters $k,p,q,r,s$ must satisfy the relation
\begin{equation}
\label{rel}
p=kr,\quad 2q+1=k(2s+1).
\end{equation}
Indeed, using the description of the restricted roots given above and the definition of the deformed root system of $BC(n,m)$ type from \cite{SV}, one can check that $n=\min(n_1,n_2),\, m=\min (m_1,m_2)$ and the parameters 
$$k=-1, \,\, p=|m_1-m_2|-|n_1-n_2|=-r, \,\, q=s=-1/2$$
for the symmetric pairs
$X=(\frak{gl} (n_1+n_2, m_1+m_2), \frak{gl}(n_1,m_1)\oplus\frak{gl}(n_2,m_2)),$ 
$$k=-\frac{1}{2}, \,\, p=|m_1-m_2|-\frac{1}{2}|n_1-n_2|=-\frac{1}{2}r, \,\, q=0, \,\, s=-\frac{3}{2}$$
for the pairs
$X=(\frak{osp} (n_1+n_2, 2m_1+2m_2), \frak{osp}(n_1,2m_1)\oplus\frak{osp}(n_2,2m_2)).$ 
For the symmetric pairs $X=(\frak{osp} (2l, 2m), \frak{gl}(l,m))$ 
we have two different cases depending on the parity of $l$:
when $l=2n$ then
$$k=-2, \,\, p=0=r, \,\, q=s=-\frac{1}{2},$$
and when $l=2n+1$ then
$$k=-2, \,\, p=-2, \,\, r=1, \,\, q=s=-\frac{1}{2}.$$

In the rest of the paper we will restrict ourselves to the case of symmetric pairs
$X=(\frak{gl} (n, 2m), \frak{osp} (n,2m)).$ In particular, we will show that the radial part homomorphism maps the centre of the universal enveloping algebra of $\frak{gl} (n, 2m)$ to the algebra of the deformed CMS integrals  $\mathcal D_{n,m}.$

\section{Symmetric pairs $X=(\frak{gl} (n, 2m), \frak{osp} (n,2m))$}

Recall that the Lie superalgebra $\frak g=\mathfrak{gl}(n,2m)$ is the sum 
$\frak g = \frak g_0 \oplus \frak g_1$, where
$\mathfrak{g}_{0}=\mathfrak{gl}(n)\oplus
\mathfrak{gl}(2m)$ and $\mathfrak{g}_{1}=V_{1}\otimes V_{2}^*\oplus
V_{1}^*\otimes V_{2},$ where $V_{1}$ and $V_{2}$ are the identical
representations of $\mathfrak{gl}(n)$ and $\mathfrak{gl}(2m))$
respectively. As a Cartan subalgebra $\frak h \subset \frak g_0$ we choose the diagonal matrices.

A bilinear form $(\, ,\, )$ on a $\mathbb Z_2$-graded vector space $V = V_0 \oplus V_1$ is said to be {\it even}, if $V_0$ and $V_1$ are orthogonal with respect to this form and it is called to be {\it supersymmetric} if
\[
(v,w)=(-1)^{p(v)p(w)}(w,v)
\]
for all homogeneous elements $v$, $w$ in $V$.
If $V$ is endowed with an even non-degenerate supersymmetric form $(\, ,\, )$, then the 
involution $\theta$ is defined by the relation
\begin{equation}
\label{theta} (\theta(x)v,w)+(-1)^{p(x) p(v)}(v,xw)=0, \quad x \in \mathfrak{gl}(V).
\end{equation}
When $\dim V_0=n, \, \dim V_1=2m$ and the form $(\, ,\, )$ coincides with Euclidean structure on $V_0$ and symplectic structure on $V_1$ we can define the {\it orthosymplectic Lie superalagebra} $\mathfrak{osp}(n,2m) \subset \mathfrak{gl}(n,2m)$ as
\[
\mathfrak{osp}(n,2m)=\{x\in \mathfrak{gl}(n,2m): \,  \theta(x)=x\}.\ 
\]

Let $\varepsilon_{1},\dots,\varepsilon_{n+2m} \in \frak h^*$  be the
weights of the identical representation of  $\mathfrak{gl}(n,2m)$.
It will be convenient also to introduce $\delta_{p}:=\varepsilon_{p+n},\: 1\le p\le 2m.$ 

The root system of $\mathfrak{g}$ is $R=R_0 \cup R_1$, where
$$R_{0}=\{\varepsilon_{i}-\varepsilon_{j}, \delta_{p}-\delta_{q} :
\: i\ne j\,:\: 1\le i,j\le n\ , p\ne q,\, 1\le p,q\le 2m\},$$ $$
R_{1}=\{\pm(\varepsilon_{i}- \delta_{p}), \quad 1\le i\le n, \,
1\le p\le 2m\}$$ 
are even and odd (isotropic) roots respectively.
We will use the following distinguished system of
simple roots $$
B=\{\varepsilon_{1}-\varepsilon_{2},\dots,\varepsilon_{n-1}-\varepsilon_{n},
\varepsilon_{n}- \delta_{1},\delta_{1}-
\delta_{2},\dots,\delta_{2m-1}-\delta_{2m}\}. $$ 

The invariant bilinear form is
determined by the relations $$
(\varepsilon_{i},\varepsilon_{i})=1,\:
(\delta_{p},\delta_{p})=-1$$ 
with all other products to be zero.
The integral weights are
\begin{equation}
\label{pogl}
 P_{0}=\{\lambda\in\mathfrak{h}^{*}\mid
\lambda=\sum_{i=1}^{n+2m}
\lambda_{i}\varepsilon_{i} =\sum_{i=1}^{n}
\lambda_{i}\varepsilon_{i}+\sum_{p=1}^{2m}\mu_{p}\delta_{p},\:
\lambda_{i}, \mu_j \in \mathbb
Z\} .
\end{equation}
The Weyl group $W_0=
S_{n}\times S_{2m}$ acts on the weights by separately
 permuting $\varepsilon_{i},\; i=1,\dots, n$ and
$\delta_{p},\; p=1,\dots, 2m.$

The involution $\theta$ is acting on $\frak h^*$ by mapping $\delta_{2j-1} \to -\delta_{2j}, \, \delta_{2j} \to -\delta_{2j-1},$
$ j=1,\dots, m$ and $\varepsilon_i \to -\varepsilon_i, \, i=1,\dots, n.$
The dual $\frak a^*$ of Cartan subspace $\frak a$ 
can be described as the $\theta$ anti-invariant subspace of $\frak h^*$
$$
\frak a^*=\{\theta(x)-x, \,\,\, x \in \frak h^*\}
$$
and is generated by 
\begin{equation}
\label{tildee}
\tilde\varepsilon_i=\varepsilon_i, \, i=1,\dots, n, \quad \tilde \delta_j=\frac12(\delta_{2j-1}+\delta_{2j}),\, \, j=1,\dots, m.
\end{equation}
The induced bilinear form in this basis is diagonal with
$$
(\tilde \varepsilon_{i},\tilde \varepsilon_{i})=1,\,\,\,
(\tilde \delta_{p},\tilde \delta_{p})=-\frac12.
$$

Let us introduce the following superanalogue of Gelfand invariants \cite{Molev}
\begin{equation}
\label{Gel}
Z_s=\sum_{i_1, \dots, i_s}^{n+2m} (-1)^{p(i_2)+\dots +p(i_{s})}E_{i_1 i_2}E_{i_2 i_3}\dots E_{i_{s-1} i_{s}}E_{i_{s} i_{1}}, \,\, s \in \mathbb N,
\end{equation}
where $E_{ij}, i,j =1, \dots, n+2m$ is the standard basis in $\frak{gl} (n,2m).$
One can define them also as $Z_s=\sum_{i=1}^{n+2m}E_{ii}^{(s)}$, where elements $E_{ij}^{(s)}$ are defined recursively by
\begin{equation}\label{recc}
E_{ij}^{(s)}=\sum_{l=1}^{n+2m}(-1)^{p(l)}E_{il}E_{lj}^{(s-1)} 
\end{equation}
with
$
E_{ij}^{(1)}=E_{ij}.
$
One can check that these elements satisfy the following commutation relations
$$
[E_{ij},E_{st}^{(l)}]=\delta_{js}E^{(l=1)}_{it}-(-1)^{(p(i)+p(j))(p(s)+p(t))}\delta_{it}E^{(l)}_{sj},
$$
which imply that the elements $Z_l$ are central. 

Let as before $\mathcal Z(X) \subset U(\frak g)^*$ be the subspace of zonal (two-sided $\frak k$-invariant) functions and $S(\frak a)^*, \, S(\frak a)^*_{loc}$ be as in the previous section.



Let $i^*: \mathcal Z(X) \to S(\frak a)^*_{loc}$ be the restriction homomorphism induced by the embedding $i: \frak a \to \frak g.$

\begin{thm}\label{radial}
The restriction homomorphism $i^*$ is injective and  there exists a unique homomorphism $\psi: Z(\frak g)\to \mathfrak D_{n,m}$ such that the following diagram is commutative
\begin{equation}
\label{commutdun1}
\begin{array}{ccc}
{\mathcal Z(X)}&\stackrel{L_z}{\longrightarrow}&{\mathcal Z(X)}\\ \downarrow
\lefteqn{i^*}& &\downarrow \lefteqn{i^*}\\
S(\frak a)^*_{loc}&\stackrel{\psi(z)}{\longrightarrow}& 
S(\frak a)^*_{loc}.\\
\end{array}
\end{equation}
where $L_z$ is the multiplication operator by $z \in Z(\frak g).$  The image of Gelfand invariants (\ref{Gel}) are the deformed CMS integrals (\ref{dif2}):
$$
\psi( Z_s)= 2^s \mathcal L_s.
$$
\end{thm}

We will call $\psi$  the {\it radial part homomorphism}. For the Casimir element $\frak C$
the operator $\psi(\frak C)$ is the deformed CMS operator (\ref{Lrad}).

\begin{proof}

Let us prove first that for any  $z\in Z(\frak g)$ there exists not more than one element $\psi(z)\in \mathfrak D_{n,m}$ which makes the diagram commutative. It is enough to prove  this only when $z=0$. 
Therefore we need to prove the following statement: if $D\in \frak D_{n,m}$ and   $D(i^*(f))=0$ for any $f\in \mathcal Z(X)$ then  $D=0$.  

Let $\frak A_{n,m} \subset C(\frak a)=\Bbb C[x_1^{\pm1},\dots, x_n^{\pm1},y_1^{\pm1}\dots,y_m^{\pm1}]$ be the subalgebra consisting of  $S_n\times S_m$-invariant Laurent polynomials $f\in C(\frak a),$
satisfying the quasi-invariance conditions (\ref{quasi}).


Let us take $f=\phi_\lambda(x) \in \mathcal Z(X)$ from Proposition \ref{restweyl}, where $\lambda \in P^+(X)$ satisfies Kac condition (\ref{Kac2}).
By Proposition \ref{restweyl} (which proof is independent from the results of this section) $i^*(f) \in \frak A_{n,m}$
and by Proposition \ref{sup}
$$
D(i^*(f))=\varphi(D)(\lambda)e^{\lambda}+\dots,
$$
where $\dots$ mean lower order terms in partial order (\ref{partial}) and $\varphi$ is the Harish-Chandra homomorphism.
If $D(i^*(f))=0$ then  $\varphi(D)(\lambda)=0$ for all $\lambda$ which are admissible  and $K(\lambda)$ is irreducible. 
By Proposition \ref{restweyl} the set of such $\lambda$ is dense in Zarisski topology in $\frak a^*$. Therefore $\varphi(D)=0$ and since the Harish-Chandra homomorphism is injective we have $D=0$.

Now let us prove that for every $z \in Z(\frak g)$ element $\psi(z)$ indeed exists. It is enough to prove this only for the Gelfand generators $Z_s$. Actually we prove now that $\psi(Z_s)=2^s\mathcal L_s$, where $\mathcal L_s$  are the deformed CMS integrals defined by (\ref{dif2}).

Let $Y= \frak k U(\frak g)+ U(\frak g) \frak k$ and  $T_{\alpha}, R_{\alpha}: U(\frak a) \to U(\frak a)$ be defined as in the previous section
 for any root $\alpha \in \frak h^*$ of $\mathfrak g$. 
For $\alpha= \varepsilon_i-\varepsilon_j$ choose  $X_{\alpha}=E_{ij}\in \frak g_{\alpha}$ and define  
$$X_{\alpha}^{(s)}=E_{ij}^{(s)},\,\,h_{\alpha}^{(s)}=[X_{\alpha},X_{-\alpha}^{(s)}], \,\, s \in \mathbb N,
$$
where $E_{ij}^{(s)}$ are given by (\ref{recc}).

\begin{lemma}\label{comm1} For any $u \in U(\frak a)$ the following equalities hold true:

{\em i)}  $R_{\alpha}^- uX_{\alpha}X^{(s)}_{-\alpha}-T_{\alpha} u
[X_{\alpha}^+, (X^{(s)}_{-\alpha})^-]\in Y$

{\em ii)} $[X_{\alpha}^+,(X^{(s)}_{\alpha})^-]=\frac12(h_{\alpha}^{(s)})^-$,
where
$$
h_{\alpha}^{(s)}=[X_{\alpha},X^{(s)}_{-\alpha}]=E_{ii}^{(s)}-(-1)^{p(i)+p(j)}E_{jj}^{(s)}.
$$

\end{lemma}
Proof is by direct calculation.

Let us denote the image $\tilde {\mathcal Z}(X) = i^*(\mathcal Z(X)) \subset S(\frak a)^*$ and define the  operators 
 $D^{(s)}_{\alpha}, \, \partial^{(s)}_{\alpha}, \partial^{(s)}_{i}:  \tilde {\mathcal Z}(X) \to S(\frak a)^*$ by the  relations
 $$D^{(s)}_{\alpha}(i^*(f))(u)=f\left(uX_{\alpha}X^{(s)}_{-\alpha}\right), \, 
\partial^{(s)}_{\alpha}(i^*(f))(u)= f\left(u h^{(s)}_{\alpha}\right),$$
$$ \partial_i^{(s)}(i^*(f))(u)=f(E_{ii}^{(s)}u) 
$$ 
for any $f \in \mathcal Z(X),$ $u\in S(\frak{a}).$ Since $i^*$ is injective these operators are well-defined.

\begin{lemma}\label{rad} For any root $\alpha=\varepsilon_i-\varepsilon_j$ of $\frak g$ the operators $D^{(s)}_{\alpha}, \, \partial^{(s)}_{\alpha}$ in the formal power series realisation satisfy the relation 
$$
(e^{\alpha}-e^{-\alpha})D^{(s)}_{\alpha}=e^{\alpha}\partial^{(s)}_{\alpha},\,\,\partial_{\alpha}^{(s)}=\partial_i^{(s)}-(-1)^{p(i)+p(j)}\partial_{j}^{(s)}.
$$
\end{lemma}

\begin{proof} Since the restriction of $f \in \mathcal Z(X)$ on $Y$ vanishes, from Lemma \ref{comm1} it follows that
$$
\hat f(R_{\alpha}^- uX_{\alpha}X_{-\alpha}^{(s)})=\hat f(T_{\alpha} u
[X_{\alpha}^+, (X_{-\alpha}^{(s)})^-]) =\frac12 \hat f(T_{\alpha} u
(h_{\alpha}^{(s)})^-)$$
$$= \frac12e^\alpha \hat f(u (h_{\alpha}^{(s)})^-)= \frac12 e^\alpha \partial^{(s)}_{\alpha} \hat f(u).
$$
Since $R_{\alpha}^- =\frac 12(T_{\alpha}-T_{-\alpha})$ we have
$$\hat f(R_{\alpha}^- uX_{\alpha}X_{-\alpha}^{(s)})=\frac 12(e^{\alpha}-e^{-\alpha})\hat f(uX_{\alpha}X_{-\alpha}^{(s)})=\frac 12(e^{\alpha}-e^{-\alpha})D^{(s)}_{\alpha}\hat f(u),$$ which implies the claim.
\end{proof}

Now from the recurrence relation (\ref{recc}) and lemma we have
$$
\partial_i^{(s)}=(-1)^{p(i)}\partial_i^{(1)}\partial_i^{(s-1)}+\sum_{j\ne i}(-1)^{p(j)}\frac{e^{2\alpha}}{e^{2\alpha}-1}(\partial_i^{(s-1)}-(-1)^{p(i)+p(j)}\partial_{j}^{(s-1)}),
$$
$i=1, \dots, n+2m,$ where $\alpha=\varepsilon_i-\varepsilon_j.$

Define the operators $\hat \partial_i^{(s)}=(-1)^{p(i)}\partial_i^{(s)}$, then the new operators satisfy the recurrence relation
$$
\hat \partial_i^{(s)}=\hat \partial_i^{(1)}\hat \partial_i^{(s-1)}-\sum_{j\ne i}(-1)^{1-p(j)}\frac{e^{2\alpha}}{e^{2\alpha}-1}(\hat \partial_i^{(s-1)}-\hat \partial_{j}^{(s-1)}), \,\, i=1,\dots, n+2m.
$$

After the restriction to $\frak a$ we have  
$$
\delta_{2j-1}=\delta_{2j}=\tilde\delta_j,\,\varepsilon_i=\tilde\varepsilon_i,\,\partial_{n+2j-1}=\partial_{n+2j}.
$$
Let us introduce 
$$
x_i=e^{2\tilde \varepsilon_i}, i=1,\dots, n, \quad y_j=e^{2\tilde \delta_j}, j=1,\dots, m.
$$
From the above recurrence relations we have  $\partial_{n+2j-1}^{(s)}=\partial_{n+2j}^{(s)}$.
We also have
$$
\partial_i(x_i)=\partial_{i}(e^{2\tilde{\varepsilon_i}})=2\tilde\varepsilon_{i}(E_{ii})e^{2\tilde{\varepsilon_i}}=2x_i,\,i=1,\dots,n,
$$
$$
\partial_{n+2j}(y_{j})=\partial_{n+2j}(e^{2\tilde\delta_j})=2\tilde\delta_{j}(E_{n+2j})e^{2\tilde{\delta_j}}=y_j,\,j=1,\dots, m.
$$

Therefore if we set $x_{n+j}=y_{j}, \, j=1,\dots,m$ we will have
$$
\hat \partial_i^{(s)}=\hat \partial_i^{(1)}\hat \partial_i^{(s-1)}-\sum_{j\ne i}^{n+m}(-1)^{1-p(j)}\frac{2^{p(j)}x_i}{x_{i}-x_j}(\hat \partial_i^{(s-1)}-\hat \partial_{j}^{(s-1)}),$$
and $\partial^{(1)}_i=2k^{p(i)}x_i\frac{\partial}{\partial x_i}, \, i=1,\dots n+m$ with
$k=-1/2.$
 So, if we define $\tilde\partial^{(s)}_i=2^{-s}\hat\partial^{(s)}_i$ we will have
 $$
\tilde \partial_i^{(s)}=\tilde \partial_i^{(1)}\tilde \partial_i^{(s-1)}-\sum_{j\ne i}^{n+m}k^{1-p(j)}\frac{x_i}{x_i-x_j}(\tilde \partial_i^{(s-1)}-\tilde \partial_j^{(s-1)}),\, \tilde\partial^{(1)}_i=k^{p(i)}x_i\frac{\partial}{\partial x_i},
$$
where $k=-1/2.$
This last recurrence relation and initial conditions coincide with (\ref{dif1}) for  this $k.$ So we have
$$
\psi(Z_s)=\sum_{i=1}^{n+2m}\partial_{i}^{(s)}=\sum_{i=1}^{n+2m}(-1)^{p(i)}\hat\partial_{i}^{(s)}=2^s\sum_{i=1}^{n+m}(-\frac12)^{-p(i)}\tilde\partial_{i}^{(s)},
$$
which coincides with formula (\ref{dif2}) with $k=-1/2$. Therefore we proved that $\psi(Z_s)=2^s \mathcal L_s \in\mathcal D_{n,m}$.

So it is only left to prove that $\psi$ is a homomorphism. Let $\psi(z_1)=D_1,\,\psi(z_2)=D_2$ then $D_1D_2$ makes the diagram commutative for $z_1z_2$. From the unicity it follows that $\psi(z_1z_2)=D_1D_2$. Theorem is proved.
\end{proof}

\section{Spherical modules and zonal spherical functions}

Let $X=(\frak g, \frak k)$ be symmetric pair and $\mathcal Z(X) \subset U(\frak g)^*$ be the set of 
the corresponding zonal functions. 

We will call such function $f \in  \mathcal Z(X)$ {\it zonal spherical function} if it is an eigenfunction of the action of the centre $Z(\frak g)$ on $U(\frak g)^*$ (cf. \cite{Hel}). 
Recall that the action of $U(\frak g)$ on $U(\frak g)^*$ is defined by the formula 
\begin{equation}
\label{leftaction}
(y l)(x)=(-1)^{p(y)p(l)}l(y^t x), \,\, x,\, y \in U(\frak g),
\end{equation}
where $y^t$ is the principal anti-automorphism of $U(\frak g)$ uniquely defined by the condition that
$y^t=-y, \, y \in \frak g.$

If $f$ is a generalised eigenfunction of $\mathcal Z(X)$, we will call it {\it generalised zonal spherical function} on $X.$ The appearance of such functions is a crucial difference of the super case from the classical one.

Let now $U$ be a finite dimensional $\frak g$-module and $U^{\frak k}$ be the space of all $\frak k$-invariant vectors $u \in U$ such that $x u =0$ for all $x \in \frak k$ and additionally that $g u=u$ for all $g \in O(n) \subset OSP(n, 2m).$ The last assumption is not essential but will be convenient.
This allows us to exclude the possibility of tensor multiplication by one-dimensional representation given by Berezinian.
Consider also similar space $U^{*\frak k}$ for the dual module $U^*.$

For any $u \in U^{\frak k}$ and $l \in U^{*\frak k}$ we can consider the corresponding zonal function
$\phi_{u,l}(x) \in \mathcal Z(X) \subset U(\frak g)^*$ defined by 
\begin{equation}
\label{fi}
\phi_{u,l}(x):=l(x u), \quad x \in U(\frak g).
\end{equation}
 We denote the linear space of such functions for given $U$ as $\mathcal Z(U).$
 
\begin{definition}
A finite dimensional $\frak g$-module $U$ is called {\it spherical} if the space $U^{\frak k}$ is non-zero.
\end{definition}

The following result should be true in general, but we will prove it only in the special case for the symmetric pair $X=(\frak{gl} (n, 2m), \frak{osp} (n,2m)).$ 

So from now on we assume that 
$\frak g = \frak{gl} (n, 2m), \, \frak k = \frak{osp} (n,2m).$

\begin{thm}\label{irrz} Let  $U$ be   irreducible spherical module.
Then $\dim \mathcal Z(U)=1$ and the corresponding function (\ref{fi}) is  zonal spherical.
\end{thm}

\begin{proof}
Let $\theta$  be the automorphism of $\frak g= \frak{gl}(n, 2m)$ defined in the previous section. It is easy to check that the correspondence  
$$
F(U)=U^{\theta},\,\, x\circ u=\theta(x)u
$$
defines a functor on the category of finite dimensional $\frak g$-modules.
From the definition (\ref{theta}) it follows that for the standard representation we have $U^{\theta}=U^*$.  It turns out that this is true for any irreducible module.

\begin{proposition}\label{dual} For any finite dimensional irreducible   $\frak g$-module $U$
$$
U^{\theta}=U^*.
$$
\end{proposition}

\begin{proof} 
In the case of Lie algebras the proof uses the fact that the longest element of the corresponding Weyl group $W$ maps the Borel subalgebra to the opposite one. In our case we do not have proper Weyl group, so we have to choose a special Borel subalgebra to prove this.

Let $\theta$ and $V$ be the same as in the previous section.
Choose a basis in $V$:
$$V_0=<f_1,\dots,f_n>,\,\, V_1=<f_{n+1},\dots,f_{n+2m}>
$$ 
such that
$$
(f_i,f_{n-i+1})=1, i=1,\dots, n,\, (f_{n+j},f_{2m+n-j+1})=1=-(f_{2m+n-j+1},f_{n+j}),
$$
where $j=1,\dots, m$ and other products are $0$. 

Let $\frak h\subset\frak g$ be the Cartan subalgebra consisting of the matrices, which are diagonal in this basis, and the subalgebras $\frak n_{+},\, \frak n_{-} \subset\frak g$ be the set of upper triangular and low triangular matrices respectively in the basis $(f_{n+1}, \dots, f_{n+m}, f_1,\dots, f_n, f_{n+m+1}, \dots, f_{n+2m}).$ 

Consider the automorphism $\omega$ acting by conjugation by the matrix $C$ with
$$C_{i \, n-i+1}=1, \, i=1,\dots, n, \,\,\,C_{n+j \, n+2m-j+1}=1, \, j=1,\dots, 2m$$ and all other entries being zero, corresponding to the product of two longest elements of groups $S_n$ and $S_{2m}.$

Then one can check that $\theta(\frak{h})=\frak h,\,\,\theta( \frak{n}_{+})=\frak{n}_{+}, \, \omega(\frak{h})=\frak h,\,\,\omega( \frak{n}_{+})=\frak{n}_{-},$ and $\theta(h)+\omega(h)=0$ for any $h\in\mathfrak h.$

Let $v\in U$  be highest weight vector with respect to Borel subalgebra $\frak h \oplus \frak n_+$ and let $\lambda\in\frak h^*$ be its weight. Then for module $U^{\theta}$ we also have $\frak{n}_{+}\circ v=\theta(\frak{n}_{+}) v=\frak{n}_{+} v=0$ and $\theta(\lambda)\in\frak h^*$ is the weight $v$ in $U^{\theta}$. 

Let $u \in U$ be the highest vector with respect to Borel subalgebra $\frak h \oplus \frak n_-,$
then the corresponding weight is $\omega(\lambda)$ (cf. \cite{Bourbaki}, Ch. 7, prop. 11).

Let $u^{*}\in U^*$ be the linear functional such that $u^{*}(u)=1$ and $u^{*}(u')=0$ for any other eigenvector  $u'$ of $\frak h$ in $\in U$. Then it is easy to see that $n_{+}u^{*}=0$ and its weight with respect to Cartan subalgebra $\frak h$ is $-\omega(\lambda)$.

 So we see that both irreducible modules $U^{\theta}$ and $U^*$ have the same highest weights $\theta(\lambda)=-\omega(\lambda)$ with respect  to the same Borel subalgebra. Therefore they are isomorphic. 
 \end{proof}

\begin{corollary}\label{form} Every finite dimensional irreducible $\frak g$-module $U$ has even non-degenerate bilinear form $( \, , \, )$ such that
$$
(\theta(x)u,v)+(-1)^{p(x)p(u)}(u,xv)=0,\,\, u,v\in U, \, x \in \frak g.
$$
The $\frak g$-modules $U$ and $U^{*}$ are isomorphic as $\frak k$-modules.
\end{corollary}

Let us choose now the standard Borel subalgebra $\frak b$ consisting of upper triangular matrices in the basis  $e_1,\dots, e_{n+2m}$ in $V$ such that $$V_0=<e_1,\dots,e_n>, \,\, V_1=<e_{n+1},\dots,e_{n+2m}>,$$ 
$(e_i,e_i)=1, \, i=1,\dots, n,\,\, \,\, (e_{n+2j-1},e_{n+2j})=-(e_{n+2j},e_{n+2j-1})=1,\, \, j=1,\dots, m$ (with all other products being zero).

For every $\frak g$ module $U$ we will denote by $U^{\frak k}$ the subspace of $\frak k$-invariant vectors.

\begin{proposition}\label{dim} 
For every  irreducible finite dimensional $\frak g$-module   $U$  we have $$\dim U^{\frak k}\le1.$$
\end{proposition}

\begin{proof} \footnote{A different proof in the general even type case can be found in \cite{AldSch}.}
Prove  first that  the map 
$$
\frak k\times\frak b\longrightarrow \frak{g},\, (x,y)\rightarrow x+y
$$
is surjective.
The kernel of this map coincides with the set of the pairs $(x,-x),x\in\frak k\cap\frak b,$
which is the linear span of vectors
$$
E_{n+2j-1,n+2j-1}-E_{n+2j,n+2j}, \,\,\,\,  E_{n+2j-1,n+2j},  \,\, j=1,\dots, m$$
and has the dimension $2m$.  Since
$$
\dim \frak k+\dim\frak b-2m=\frac12n(n-1)+m(2m+1)+2nm
$$
$$
+\frac12(n+2m)(n+2m+1)-2m=(n+2m)^2=\dim \frak g,
$$
which implies the claim.  

Let $v \in U$ be the highest weight vector with respect to Borel subalgebra $\frak b$ and 
$w$ be vector invariant with respect to $\frak k$: $y w=0$ for all $y \in \frak k$,   and such that 
$(w,v)=0.$ We claim that this implies that $w =0.$

To show this we prove by induction in $N$ that  
$
(w, x_1x_2\dots x_N v)=0$  
for all  $x_1, \dots, x_N \in \frak g.$
This is obviously    true  when $N=0$. Suppose that this is true for  $N$.  Take any $x\in\frak g$ and represent it in the form $x=y+z,\,y\in \frak{k},\, z\in \frak b,$ so that
$$
(w, xx_1x_2\dots x_N v)=(w, yx_1x_2\dots x_N v)+(w, zx_1x_2\dots x_N v).
$$
By definition we have $(w, yx_1x_2\dots x_N v)=\pm(yw, x_1x_2\dots x_N)=0.$
We have 
$$
(w, zx_1x_2\dots x_N v)=(w, [z,x_1x_2\dots x_N] v)\pm (w, x_1x_2\dots x_N z v)=0.$$
By inductive assumption
$$
(w, [z,x_1x_2\dots x_N] v)=(w, [z,x_1]x_2\dots x_N v)\dots\pm(w, x_1x_2\dots [z,x_N] v)=0
$$
and since $zv = cv$ for any $z \in \frak b$ we also have $(w, x_1x_2\dots x_N z v)=0.$
Thus we have $(w, u)=0$ for any $u\in U$, so $\omega =0$ since  the form is non-degenerate.

 Let  now $w_1, w_2$ be two $\frak k$-invariant non-zero vectors. Then we have 
 $$
(w_1,v)=c_1\ne0,\, (w_2,v)=c_2\ne0,
$$
so $(c_1w_1-c_2w_2,v)=0$ and thus $c_1w_1=c_2w_2$. Thus the dimension of the space $U^{\frak k}$ of $\frak k$-invariant vectors in $U$ can not be greater than 1.
\end{proof}

Now let us deduce the Theorem. Let $u \in U^{\frak k}$ be a non-zero vector, then by Corollary \ref{form} there exists a non-zero $l \in  U^{*\frak k}.$ By Proposition \ref{dim} they are unique up to a multiple, therefore $\dim \mathcal Z(U)\leq1.$ So we only need to show that $\mathcal Z(U) \neq 0.$

From the proof of Proposition \ref{dim} it follows that $l(v)\neq 0$ for a highest vector $v \in U$.
Since module $U$ is irreducible $v = xu$ for some $x \in U(\frak g)$, so $\phi_{u,l} \neq 0$ and thus $\mathcal Z(U) \neq 0.$

This proves that if $U^{\frak k} \neq 0$ then $U$ is spherical. The converse is trivial.

Since the space $\mathcal Z(U)$ is one-dimensional and centre $Z(\frak g)$ preserves it, it follows that  the function  $\phi_{u,l}$ is zonal spherical.
\end{proof}

Now we would like to describe the conditions on the highest weights for irreducible modules to be spherical. 
 
Let $\varepsilon_i$ be the basis in $\frak h^*$  dual to the basis $E_{ii},\,i=1,\dots,n+2m$. Let us call the weight $$\lambda=\sum_{i=1}^{n+2m}\lambda_i\varepsilon_i \in \frak h^*$$ {\it admissible} for symmetric pair $X$ if 
$$
\lambda_i\in2\Bbb Z,\,i=1,\dots,n,\,\,\,\,\lambda_{n+2j-1}=\lambda_{n+2j}\in\Bbb Z,\,j=1,\dots,m.
$$
Denote the set of all such weights as $P(X)$.
Let also $P^+(X) \subset P(X)$ be the subset of highest admissible weights:
$$
\lambda_1\ge\dots\ge\lambda_{n},\,\,\lambda_{n+1}\ge\dots\ge\lambda_{n+2m}.
$$

Let $U=L(\lambda)$ be a finite-dimensional irreducible module with highest weight $\lambda$ and $U^*=L(\mu)$.  Proposition \ref{restweyl} implies that both $\lambda$ and $\mu$ are admissible.
We conjecture that this condition is also sufficient.

{\bf Conjecture.} {\it If the highest weights $\lambda$ and $\mu$ of both $U$ and $U^*$  are admissible then $U$ is spherical.}

We will prove this only under additional assumption of typicality, which is a natural generalisation of Kac's typicality conditions for Kac modules (see Corollary \ref{corcon} below).

Let us remind the notion of Kac module  \cite{Kac2}.
Let $\frak g_0$  be the even part of the Lie superalgebra $\frak g$ and 
$$\frak p=\frak p_0\oplus \frak p_1,$$ where $\frak p_0=\frak g_0$ and
$\frak p_1 \subset \frak g_1$ be the linear span of  positive odd root subspaces. 

 Let $V^{(0)}$ be irreducible finite dimensional $\frak g_0$-module. 
 Define the structure of $\frak p$-module on it by setting 
$\frak p_1 V^{(0)}=0.$ The Kac module is defined as an induced module by
 $$
 K(V^{(0)})= U(\frak g)\otimes_{U(\frak p)}V^{(0)}
 $$
If $V^{(0)}=L^{(0)}(\lambda)$ is the highest weight $\frak g_0$-module with weight $\lambda,$ then the corresponding Kac module is denoted by $K(\lambda)$.

The following theorem describes the main properties of the Kac modules.

Recall that $\frak k = \frak{osp}(n, 2m) \subset \frak g=\frak{gl}(n, 2m).$

\begin{thm}\label{Kac}

$1)$ We have the isomorphism of $\frak k$-modules  
$$
K(\lambda)=U(\frak k)\otimes_{U(\frak k_0)} L^{(0)}(\lambda).
$$

$2)$ $K(\lambda)$ is projective as $\frak k$ module.

$3)$  As $\frak g$ modules 
$$
K(\lambda)^*=K(2\rho_1-w_0(\lambda)),
$$
where $2\rho_1$ is the sum of odd positive roots and $w_0$ is the longest element of the Weyl group $S_n\times S_{2m}$.

$4)$  $K(\lambda)$ is spherical if and only if $\lambda\in P^+(X)$,  in which case $$\dim K(\lambda)^{\frak k}=1.$$
\end{thm} 

\begin{proof} We start with the following important fact,  which can be easily checked:
 $$(1+\theta)\frak p_{1}=\frak{k}_{1},$$
where $\frak k_1$ is the odd part of $\frak k.$  Let 
$$
\varphi : U(\frak k)\otimes_{U(\frak k_0)} L^{(0)}(\lambda) \rightarrow K(\lambda)
$$
 be the homomorphism of $\frak k$-modules induced by natural inclusion $L^{(0)}(\lambda) \subset K(\lambda)$. 
 Let $\frak p_{-1}$ be the the linear span of negative odd root subspaces and consider  the filtration on   $K(\lambda)$ such that
$$
L^{(0)}(\lambda)=K_{0}\subset K_{1}\subset\dots\subset K_{N-1}\subset K_{N}=K(L^{(0)}(\lambda)),
$$
  where $N=\dim \frak p_{-1}$ and  $K_{r}$ is the linear span of $x_{1}\dots x_{s}v, \, v \in L^{(0)}(\lambda)$, where  $x_{1},\dots, x_{s}\in \frak p_{-1},\,\,s\le r$.
 
 Now let us prove by induction in $r$ that $K_{r}\subset  Im \,\varphi$.
 Case when $r=0$ is obvious.  Let  $K_{r}\subset Im \,
\varphi $, then  $(x+\theta (x))K_{r}\subset
Im \, \varphi $ for all $x \in \frak p_{-1}$ since $x+\theta(x) \in \frak k.$
Since $\theta (\frak p_{-1})= \frak p_1$ we have $\theta (x)K_{r}\subset K_{r-1}$. 

Therefore $xK_{r}\subset Im \, \varphi $ and thus 
$K(\lambda)=K_N\subset Im \, \varphi $. This means that the homomorphism $\varphi$ is surjective. Since both modules have the same dimension $\varphi$ is an isomorphism. This proves the first part of the theorem.

Part $2)$ now follows since every induced module from $\frak p_0$ to $\frak p$ is projective, see \cite{Zou}. 
Part $3)$ can be found in Brundan \cite{Bru1}, see formula (7.7).

So we only need to prove part  $4)$. From the first part we have the isomorphism of the vector spaces 
$$(K(\lambda)^*)^{\frak k}=(L^{(0)}(\lambda)^*)^{\frak k_0}.
$$
Thus we reduced the problem to the known case of Lie algebras. In particular, according to 
\cite{GW} $$\dim (L^{(0)}(\lambda)^*)^{\frak k_0} = 1$$ if  $\lambda\in P^{+}(X)$ and $0$ otherwise.
From part $3)$ it follows that the same is true for the module $K(\lambda).$
\end{proof} 

We need some formula for the zonal spherical functions related to irreducible modules.

 Let $W(X)=S_n\times S_m$ be the {\it restricted Weyl group}. It acts naturally on $\frak a^*$  
 permuting $\tilde \varepsilon_i, \tilde \delta_j$ given by (\ref{tildee}). Let
$$C(\frak a)=\Bbb C[x_1^{\pm1},\dots, x_n^{\pm1},y_1^{\pm1}\dots,y_m^{\pm1}]$$  be the subalgebra of $U(\frak a)^*$, where $x_i=e^{2\tilde \varepsilon_i}, \, y_j=e^{2\tilde \delta_j}.$
The subalgebra $\frak A_{n,m} \subset C(\frak a)$ consists of  $S_n\times S_m$-invariant Laurent polynomials $f\in C(\frak a)$
satisfying the quasi-invariance conditions (\ref{quasi}).

\begin{proposition}\label{restweyl} Let $L(\lambda)$  be an irreducible spherical  finite-dimensional module and $\phi_\lambda(x)$ be the corresponding  zonal spherical function (\ref{fi}). Then its restriction to $U(\frak a)$ belongs to $\frak A_{n,m} $ and has a form
\begin{equation}
\label{formx}
i^*(\phi_\lambda(x))=\sum_{\mu\preceq\lambda}c_{\lambda,\mu}e^{\mu}(x), \,\, x \in U(\frak a),
\end{equation}
where $\lambda \in P^+(X),\,\mu\in P(X)$ and $c_{\lambda,\lambda}=1.$ 

The restrictions to $\frak a$ of the weights $\lambda,$ for which $L(\lambda)$ is an irreducible spherical  finite-dimensional module, are dense in Zarisski topology in $\frak a^*.$
\end{proposition}

\begin{proof}
The proof of the form (\ref{formx}) and of the invariance under $W(X)$ can be reduced to the case of Lie algebras by considering $L(\lambda)$ as a module over $\frak g_0$ (see e.g. Goodman-Wallach \cite{GW}). Note that Goodman and Wallach consider the representations of Lie groups (rather than Lie algebras), so in order to use their result we need the invariance of vector under $O(n),$ which is assumed in the definition of spherical modules.

To prove the quasi-invariance conditions  
we use the fact that $i^*(\phi_\lambda)$ is an eigenfunction of the Laplace-Beltrami operator on $X$.
Since the radial part of this operator has the form (\ref{radpart}) we see that the derivative 
$\partial_{\alpha}\phi_\lambda$ must vanish when $e^{2\alpha}-1=0$ for all roots $\alpha$, which implies the conditions (\ref{quasi}) in the variables $x_i, y_j.$


According to \cite{Kac} Kac module $K(\lambda)$ is irreducible if and only if {\it Kac's typicality conditions}
\begin{equation}
\label{Kac2}
\prod_{i=1}^n \prod_{j=1}^{2m}(\lambda + \rho, \varepsilon_i -\delta_j) \ne 0
\end{equation}
are satisfied, where $\rho$ is given by
\begin{equation}
\label{rho2}
\rho=\frac12\sum_{i=1}^n(n-2m-2i+1)\varepsilon_i+\frac12\sum_{j=1}^{2m}(2m+n-2j+1)\varepsilon_{n+j}.
\end{equation}
Let $\lambda\in P^+(X)$ be an admissible weight, satisfying this condition. Then the corresponding Kac module $K(\lambda)$ is irreducible and spherical by Theorem \ref{Kac}.  Since such $\lambda$ are dense in $\frak a^*$ the proposition follows.
\end{proof}

\section{Spherical typicality}

Let $\frak g= \frak {gl}(n, 2m)$ and $V$ be finite-dimensional  irreducible  $\frak g$-module. By Schur  lemma any element from the centre $Z(\frak g)$ of the universal enveloping algebra $U(\frak g)$ acts as a scalar in  $V$  and therefore we have the homomorphism 
\begin{equation}
\label{central}
 \chi_{V} : Z(\frak g)\longrightarrow \Bbb C,
\end{equation}
which is called the {\it central character} of $V$.

The following notion is an analogue of Kac's typicality conditions \cite{Kac}.

\begin{definition} Irreducible finite-dimensional  spherical module $L(\lambda)$ is called  \textbf {\textit {spherically typical}} if it is uniquely defined by its central character among the spherical irreducible $\frak g$-modules.  
\end{definition}

In other words, if $L(\mu)$ is another irreducible finite-dimensional spherical  $\frak g$-module and $\chi_{\lambda}=\chi_{\mu},$ then $\lambda=\mu$.

In order to formulate the conditions on the highest admissible weight to be typical we need several results from the representation theory of Lie superalgebras \cite{Kac, PS, Serge1}. 

Let $
\varphi: Z(g)\longrightarrow S(\frak h)
$
be the Harish-Chandra homomorphism \cite{Kac}. Let $\varepsilon_i \in \frak h^*, \, i=1,\dots, n+2m$  be the same as in the previous section and $\rho$ is given by (\ref{rho2}).

\begin{thm}\label{HCh} \cite{Serge1} 
The image of the Harish-Chandra homomorphism $\varphi$ is isomorphic to the algebra of polynomials from $P(\frak h^*)=S(\frak h)$, which have the following properties:

$1)$
$
f(w(\lambda+\rho))=f(\lambda+\rho),\, w\in S_n\times S_{2m}, \lambda \in \frak h^*,
$

$2)$ for all odd roots $\alpha \in \frak h^*$ and $\lambda \in \frak h^*$ such that $(\lambda+\rho,\alpha)=0$
$$f(\lambda+\alpha)=f(\lambda).$$

It is generated by the following polynomials:
\begin{equation}\label{gener}
P_{r}(\lambda)=\sum_{i=1}^n(\lambda+\rho,\varepsilon_i)^r-\sum_{j=1}^{2m}(\lambda+\rho,\varepsilon_{n+j})^r,\,r \in \mathbb N.
\end{equation}
\end{thm}

We are going to apply this theorem in the case of admissible highest weights $\lambda \in P^+(X)$.
For any such weight  define two sets 
$$
A=\{a_1,\dots,a_n\},\quad B=\{b_1,\dots,b_m\},
$$
where
$$
a_i=(\lambda+\rho,\varepsilon_i)-\frac12(n-2m-1)=\lambda_i+1-i,\, i=1,\dots,n,
$$
$$
 b_j=(\lambda+\rho,\varepsilon_{n+2j})-\frac12(n-2m-1)=-\lambda_{n+2j}-n+2j,\, j=1,\dots, m.
 $$
It is easy to check that this establishes a bijection between  the set $P^+(X)$ of the highest admissible weights and  the set $\mathcal T$  of pairs  $(A,B)$, where $A, B \subset \mathbb Z$ are finite subsets of $n$ and $m$ elements respectively, which satisfy the following conditions: 

$i)$  If $a,\tilde a\in A$ and there is no any other element in $A$ between them (in the natural order in $\mathbb Z$), then $a-\tilde a$ is odd integer.

$ii)$ If we denote by $B-1$ the shift   
 of the set $B$ by $-1$ then $B\cap(B-1)=\emptyset$.
 
 The dominance partial order on the set of admissible highest weights $P^+(X)$ induces some partial order on $\mathcal T,$ which we will be denote by the same symbol $\prec$.

 The Harish-Chandra homomorphism defines an equivalence relation on the set $P^+(X):$
 $\lambda \sim \mu$ if and only if $\chi_{\lambda}=\chi_{\mu}$.
 The following lemma describes the corresponding equivalence  relation on the set $\mathcal T$.

\begin{lemma} \label{equival} Let $\lambda,\tilde\lambda$ be admissible highest weights and $(A,B),\,(\tilde A,\tilde B)$ be the corresponding elements in $\mathcal T$. 
Then $\chi_{\lambda}=\chi_{\tilde\lambda}$  if and only if  the following conditions are satisfied: 
$$
A\setminus(B\cup B-1)=\tilde A\setminus(\tilde B\cup \tilde B-1),
$$
$$
(B\cup B-1)\setminus A=(\tilde B\cup \tilde B-1)\setminus\tilde A.
$$
\end{lemma}

\begin{proof}
Let $C=B\cup(B-1), \, \tilde C=\tilde B\cup(\tilde B-1),$ then from (\ref{gener}) it follows that 
$$
\sum_{i=1}^n a_i^r - \sum_{i=1}^{2m} c_i^r = \sum_{i=1}^n \tilde a_i^r - \sum_{i=1}^{2m} \tilde c_i^r, \, \, r \in \mathbb N.
$$
Therefore 
$$
\sum_{i=1}^n a_i^r + \sum_{i=1}^{2m} \tilde c_i^r= \sum_{i=1}^n \tilde a_i^r +\sum_{i=1}^{2m} c_i^r , \, \, r \in \mathbb N.
$$
Hence the sequences $(a_1,\dots, a_n, \tilde c_1, \dots, \tilde c_{2m})$ and $(\tilde a_1,\dots, \tilde a_n, c_1, \dots, c_{2m})$ coincide up to a permutation. Hence 
$A\setminus C=\tilde A\setminus \tilde C$ and $C\setminus A=\tilde C\setminus \tilde A.$
\end{proof}
  


\begin{proposition}\label{atip} 

If $A\cap B\ne\emptyset$, then there exists $(\tilde A,\tilde B) \in \mathcal T$, which is equivalent to $(A,B),$ such that $(\tilde A,\tilde B)\prec(A,B)$. In particular, every finite equivalence class contains  a representative $(A, B)$ such that $A\cap B=\emptyset.$

If $A\cap B=\emptyset$, then the equivalence class of $(A,B)$ in $\mathcal T$ is finite and contains   $2^s$ elements, where $s=|A\cap(B-1)|$ and $(A,B)\preceq (\tilde A,\tilde B)$ for any $(\tilde A,\tilde B)\in\mathcal T$.
\end{proposition}

\begin{proof} Let $A\cap B\ne\emptyset$. Represent $B\cup B-1$ as the disjoint union of the segments of  integers
$$
B\cup (B-1)=\cup_{i}\Delta_i,
$$
where a segment  is a finite set of integers  $\Delta$ such that $a,b\in \Delta$ and $a\le c\le b$ imply $c\in\Delta$.
Since $B$ and $B-1$ do not intersect  then any $\Delta_i$ consists of even number of integers and if we set $C_i = B\cap \Delta_i$ then $C_i\cap (C_i-1)=\emptyset$ and $\Delta_i=C_i\cup (C_i-1)$.

Consider two cases. In the first case suppose that  there exist  $i$  and $a\in A \cap C_i$ such that if $a'\in \Delta_i$ and $a'<a$ then $a'\notin A$.  Let $\Delta_i=[c,d]$ and set
$$
\tilde A=(A\setminus\{a\})\cup\{c-1\},\quad \tilde B=B\setminus \{a\}\cup\{c-1\}.
$$
We need to prove that $(\tilde A,\tilde B)$ is equivalent to $(A,B)$ and that  $(\tilde A,\tilde B)\in\mathcal T$.

Indeed, it is clear that $c-1\notin B\cup (B-1)$ and since $a-(c-1)$ is an even number we have $c-1\notin A$.  Therefore $(\tilde A,\tilde B)$ is equivalent to $(A,B)$. 

It is easy to see that if $a'\in A$  is a neighbour of $a$ then the difference $a'-(c-1)$ is an odd integer.
 Further we have 
$$
\tilde B\cup (\tilde B-1)=\cup_{j\ne i}\Delta_j\cup[c-1,a-1]\cup[a+1,d],
$$
which proves that $(\tilde A,\tilde B)\in\mathcal T$. Since $c-1<a$ we have $(\tilde A,\tilde B)\prec( A, B)$ (see \cite{Brund}).

Now suppose that the conditions of the first case are not fulfilled. From the assumption $A\cap B \ne\emptyset$ we see that there exist $i$  and $a\in C_i\cap A$ such that there exists $a'\in \Delta_i \cap A$ and $a'<a$.
Since $[a',a]\subset[c,d]$ we can assume that $a',a$ are neighbours. Since the difference $a-a'$ must be odd, we have $a'\in (C_i-1)$. 

Let $a_{min}$ be the minimal element from $A$.
Choose a set $E=\{ e-1, e\}$ such that  $a_{min}-e$ is positive odd, $E\cap (B\cup(B-1))=\emptyset$ and define
$$
\tilde A=(A\setminus\{a,a'\})\cup E,\,\,\tilde B=(B\setminus\{a\})\cup \{e\}.
$$
It is easy to verify that the $(\tilde A,\tilde B) \sim (A,B)$, $(\tilde A,\tilde B)\in \mathcal T$  and $(\tilde A,\tilde B)\prec(A,B)$.
Note that in this case we have always an infinite equivalence class.

Now let us prove the second part.

Assume that $A\cap B=\emptyset$. In that case every $\Delta_i$ contains not more than one element of $A$, which must belong to $(C_i-1)$ since $A\cap C_i=\emptyset.$ Indeed, if $\Delta_i \cap A$ contains two elements $a$ and $a'$ then $[a, a'] \subset \Delta_i$, so we can assume without loss of generality that $a$ and $a'$ are neighbours. Since both of them belong to $(C_i-1)$ the difference $a'-a$ is even, which is a contradiction.
 
Let $a \in A \cap \Delta_i$ and $\Delta_i=[c,d].$  Let $(\tilde A, \tilde B) \sim (A,B), \, (\tilde A, \tilde B) \in \mathcal T.$  Then there are two possibilities: $(\tilde B \cup (\tilde B-1)$ contains $\Delta_i$ or not. In the last case the only possibility for $\Delta_i$ is to be replaced by the union $[c,a-1] \cup [a+1, d+1]$, which leads to $2^s$ possibilities.
 \end{proof} 

We now need the following Serganova's lemma  \cite{Serga,PS}, which connects  two highest weight vectors  in an irreducible module with respect to an odd reflection.

Let $\frak g=\frak{gl}(n,l)$
and $\frak b,\,\frak b'$ be two Borel subalgebras such that
$$
\frak b'=(\frak b\setminus\{\gamma\})\cup\{-\gamma\}
$$
where $\gamma=\varepsilon_p-\varepsilon_{n+q}$ is a simple odd root, and $\rho$ and $\rho'=\rho-\gamma$ be the corresponding Weyl vectors.

Let $V$ be a simple finite-dimensional $\frak g$-module and $v$ and $v'$ be the highest weight vectors with respect to $\frak b$ and  $\frak b'$ respectively. Let $\lambda$ and $\lambda'$ be the corresponding weights. 

Define the sequences
$A=\{a_1,\dots,a_n\}, B=\{b_1,\dots,b_{l}\}, A'=\{a'_1,\dots,a'_n\}, $
$B'=\{b'_1,\dots,b'_{l}\},$
where
\begin{equation}
\label{sets}
a_i=(\lambda+\rho,\varepsilon_i),\,i=1,\dots,n,\,\,\, b_j=(\lambda+\rho,\varepsilon_{n+j}),\,j=1,\dots, l,
\end{equation}
$$
a'_i=(\lambda'+\rho',\varepsilon_i),\,i=1,\dots,n,\,\,\, b'_j=(\lambda'+\rho',\varepsilon_{n+j}),\,j=1,\dots, l.
$$

\begin{lemma} \cite{Serga,PS}  If $a_p\ne b_q $ then 
$$
 A'=A,\,\, B'=B.
 $$
 If $a_p=b_q$ then 
 $$
 a'_p=a_p+1,  \,\, b'_q=b_q +1, 
 $$
and $ \, a'_i=a_i, i\ne p, \, \, b'_j=b_j, j \ne q.$
 \end{lemma}
 
 
 Define the following operation on the pairs of sequences $F: (A, B) \rightarrow (\tilde B, \tilde A)$ recursively. If $A$ and $B$ consist of one element $a$ and $b$ respectively then  
 \begin{equation}
 \label{rule}
 F(a,b)=
 \begin{cases}
 (b,a),  \, b \ne a\\
 (b+1, a+1), b=a.
 \end{cases}
 \end{equation}
 If $A=\{a_1,\dots,a_n\}$, $B=\{b_1,\dots,b_{l}\}$, then we repeat this procedure for all elements of $A$
starting with $a_n$ and moving them to the right of $B$ using the rule (\ref{rule}).

 \begin{example}
 If $A=(3,2,5), B=(3,1,2,4)$ then $$F(A,B)=(\tilde B, \tilde A)=((4,1,3,5), (5, 3, 5)).$$
 \end{example}
 
 Let $\frak b$ be the standard Borel subalgebra of $\frak {gl}(n,l)$ and $ \tilde {\frak b}$ be its "odd opposite" with the same even part and odd part replaced by the linear span of negative odd root vectors.

 \begin{proposition}\label{lowhighest} \cite{Serga,PS} Let $(A,B)$ be the sequences (\ref{sets}) corresponding to the highest weight of $\frak g$-module $V$ with respect to standard Borel subalgebra $\frak b$, then  the highest weight of $V$ with respect to the odd opposite Borel subalgebra $\tilde {\frak b}$ is $(\tilde A,\,\tilde B)$, where $(\tilde B,\,\tilde A)=F(A,B).$
   \end{proposition}

There is a natural bijection
$
\sharp: P^+(X)\longrightarrow X_{n,m}^{+}
$
mapping the admissible weight $\lambda=(2\lambda_1,\dots, 2\lambda_n,\mu_1,\mu_1,\dots,\mu_m,\mu_m)$ to 
\begin{equation}
\label{sharp}
\lambda^\sharp=(\lambda_1,\dots,\lambda_n,\mu_1,\dots,\mu_m).
\end{equation}

\begin{thm}\label{typical} A finite-dimensional irreducible $\frak g$-module  $L(\lambda)$ is spherically typical  if and only if $\lambda \in P^+(X)$ and 
\begin{equation}
\label{star}
\prod_{1\le i\le n,\, 1\le j\le m}(\lambda+\rho,\varepsilon_i-\delta_{2j})\ne0,
\end{equation}
where $\rho$ is given by (\ref{rho2}). In terms of the restricted roots $\alpha \in R(X) \subset \frak a^*$ this condition can be written in the invariant form
\begin{equation}
\label{starinv}
\prod_{\alpha \in R_+(X)}[(\lambda^\sharp+\rho(k),\alpha)-\frac12(\alpha, \alpha)]\ne 0,
\end{equation}
where $\rho(k)$ is given by (\ref{rhok}) with $k=-\frac12$ and the form on $\frak a^*$ is induced from the restricted form on $\frak a.$
\end{thm}

\begin{proof}  Let us prove first that the conditions are necessary.   

Let $L(\lambda)$ be spherically typical irreducible finite dimensional $\frak g$-module. Then  the dual module $L(\lambda)^*$  is also spherical  by corollary \ref{form}. Since $L(\lambda)$ is the homomorphic image of $K(\lambda)$  the dual Kac module $K(\lambda)^*\supset L(\lambda)^*$ is spherical. But by part $3$ of Theorem \ref{Kac}   we have $K(\lambda)^*=K(2\rho_1-w_0(\lambda))$. Therefore by part $4$ of the same theorem we have  $2\rho_1-w_0(\lambda)\in P^+(X),$ which implies that $\lambda\in P^+(X)$.

 Now let us prove that $\lambda$ satisfies the condition (\ref{star}). Suppose that this is not the case. This means that $A\cap B\ne\emptyset$ for the corresponding sets in $\mathcal T$. By  Proposition \ref{atip} there exists $(\tilde A,\tilde B)\in \mathcal T$ such that $(\tilde A,\tilde B)\prec( A, B)$ and 
 $(\tilde A,\tilde B)\sim( A, B)$. Therefore there exists $\mu\in P^+(X)$ such that $\mu\prec\lambda$ and $\chi_{\mu}=\chi_{\lambda}$.  Then  by Theorem  \ref{Kac}, part $4$  module $K(\mu)$ contains an invariant vector $\omega$.  Consider the Jordan--H\"older series of $K(\mu):$ 
$$
K(\mu)=K_0\supset K_1\supset\dots\supset K_N=0.
$$
There exists $0\le i\le N-1$ such that $\omega\in K_i,\,\omega\notin K_{i+1}$.
Therefore sub-quotient $L(\nu)=K_i/K_{i+1}$ is an irreducible spherical module and $\chi_{\nu}=\chi_{\mu}=\chi_{\lambda}$.
 Since $L(\lambda)$ is spherically typical then $\lambda=\nu\preceq\mu$. Contradiction means that (\ref{star}) is satisfied.
 
 To prove that the conditions are sufficient assume that $\lambda\in P^+(X)$ and (\ref{star})  is fulfilled. Then $A\cap B=\emptyset$. We claim that $L(\lambda)$ contains a $\frak k$-invariant vector. Indeed, since $\lambda\in P^+(X)$ Kac module $K(\lambda)$ contains a $\frak k$-invariant vector $\omega$.  Let $L(\mu)$ be an irreducible spherical sub-quotient of $K(\lambda)$, such that the image of $\omega$ in $L(\mu)$ is non-zero. As before, $\mu\in P^+(X)$, $\mu\preceq\lambda$ and  $\chi_{\lambda}=\chi_{\mu}$. But  since $A\cap B=\emptyset$ then by part 2 of Proposition \ref{atip} $\lambda\preceq\mu$. Therefore $\lambda =\mu$ and $L(\lambda)$ is spherical.

So we only need to prove that $L(\lambda)$ is spherically typical module. 

Let $L(\nu)$  be an irreducible spherical module such that $\chi_{\lambda}=\chi_{\nu}$. As we have already shown $\nu\in P^+(X)$. Suppose that $\nu\ne\lambda$.   By Corollary \ref{form} the dual module $L(\nu)^*$ also spherical. It is known that $L(\nu)^*=L(-w_0(\mu)),$ where $w_0$ is the longest element of the Weyl group $S_n\times S_{2m}$ and $\mu$ is the weight of the highest weight vector in $L(\nu)$ with respect to the  odd opposite Borel subalgebra $\tilde {\frak b}.$ 
Therefore $\mu\in P^+(X)$. 

Let $(\hat A, \hat B),\,(\tilde A,\tilde B)\in \mathcal T$ be the pairs corresponding to $\nu$ and $\mu$ respectively. By Proposition \ref{lowhighest} we have
$
F(\hat A, \hat B)=(\tilde B,\tilde A).
$
 Let us represent  $ \hat B\cup( \hat B-1)$  as the disjoint union of integer segments 
 $$
 \hat B\cup( \hat B-1)=\bigcup_{i}\Delta_i.
 $$
 We assume that the labelling is done in the increasing order: if $i'<i$ and $x'\in\Delta_{i'},\,x\in\Delta_{i}$ then $x'<x$.
Note that since $(\hat A, \hat B)\in \mathcal T$ every segment consists of even number of points.
  From the description of the equivalence class containing $(A, B)$ (see Proposition \ref{atip}) it follows that for any $i$ there are the following  possibilities: 
  
  $1)$ $\Delta_i \cap \hat A = \emptyset$,
  
  $2)$ $\Delta_i$ contains one element from $\hat A\cap (\hat B-1)$,
  
  $3)$ $\Delta_i=[c_i,d_i]$ contains one element  $ a\in \hat A\cap \hat B$ and $\hat A\cap[c_i,a-1]=\emptyset$.
    
  Let us choose $i$ such that $\Delta_i$ has property $3)$ and  $i$ is minimal.  
  Since $\hat A\cap \hat B\ne\emptyset$ such $i$ indeed does exist. Then one can check using Proposition \ref{lowhighest} that $a-1\in \tilde B\cup(\tilde B-1)$ and the segment containing this element consists of odd number of points. Therefore $(\tilde A,\tilde B)\notin\mathcal T$, which is a contradiction. Theorem is proved.
  \end{proof}
  
 \begin{remark} 
The usual typicality (\ref{Kac2}) for admissible weights implies the spherical typicality (\ref{star}), but the converse is not true.
 As it follows from the proof of the theorem the degree of atypicality \cite{Brun} of a spherically typical module $L(\lambda)$ is equal to $s$, 
 where $2^s$ is the number of elements in the equivalence class of $\lambda,$ and thus can be any number (see Proposition \ref{atip}).
 \end{remark}

\begin{corollary}\label{corcon}
If the highest weight $\lambda$ of the irreducible module $U=L(\lambda)$ is admissible and satisfies (\ref{star}) then the highest weight $\mu$ of $U^*$ is also admissible and $U$ is spherical.
\end{corollary}

We should mention that a different proof of the sphericity of $U$ under the assumption that the weight $\lambda$ is large enough was found in \cite{AldSch}.

\section{Proof of the main theorem}

Let  $\mathfrak D_{n,m}$ be the algebra of quantum integrals of the deformed CMS system with parameter $k=-\frac12.$ It  acts naturally on the algebra $\frak A_{n,m}$ of $S_n\times S_m$-invariant Laurent polynomials
$f\in\Bbb C[x_1^{\pm1},\dots, x_n^{\pm1},y_1^{\pm1}\dots,y_m^{\pm1}]^{S_n\times S_m}$
satisfying the quasi-invariance condition (\ref{quasi})
with the parameter $k=-\frac12.$ Let
$$
\frak A_{n,m}=\bigoplus_{\chi}\frak A_{n,m}(\chi)
$$
be the corresponding decomposition into the direct sum of the generalised  eigenspaces (\ref{sum}).  

On the set of highest admissible weights $P^+(X)$ there is a natural equivalence relation defined by the equality of the corresponding central characters (\ref{central}).
Under bijection (\ref{sharp}) it goes to the equivalence relation on $X^+_{n,m}$ when $\lambda \sim \mu$ if $\chi_\lambda = \chi_\mu.$ 

Consider $\frak A_{n,m}$ as $Z(\frak g)$-module with respect to the radial part homomorphism $\psi.$

\begin{thm} For any finite dimensional generalised eigenspace $\frak A_{n,m}(\chi)$ there exists  a unique projective indecomposable module $P$ over $\frak{gl}(n,2m)$ and a natural map from $\frak k$-invariant part of $P^*$
$$
\Psi : (P^*)^{\frak k}\longrightarrow \frak A_{n,m}(\chi)
$$
which is an isomorphism of $Z(\frak g)$-modules. 
\end{thm}

\begin{proof} Let $\chi :\mathfrak D_{n,m}\rightarrow \Bbb C$ be a homomorphism such that $\frak A_{n,m}(\chi)$ is a finite dimensional  vector space. By Proposition \ref{max} there exists $\nu \in X^+_{n,m}$ such that $\chi=\chi_{\nu}$.

Let $E$ be the equivalence class in  $P^+(X)$ corresponding to the equivalence class of $\nu$ via $\sharp$ bijection.  Since the corresponding equivalence class is finite by Proposition \ref{atip} and the definition of the sets $A$ and $B$ there exists $\lambda\in E,$ which satisfies condition (\ref{star}).  

By Theorem \ref{typical} the corresponding irreducible module $L(\lambda)$ is spherically typical. 
Let $K(\lambda)$ is the corresponding Kac module. Consider the Jordan--H\"older series of $K(\lambda):$ 
$$
K(\lambda)=K_0\supset K_1\supset\dots\supset K_N=0.
$$
Since $\lambda\in P^+(X)$ Kac module $K(\lambda)$ contains a non-zero $\frak k$-invariant vector $v$.  Let $L(\mu)$ be an irreducible spherical sub-quotient of $K(\lambda)$, such that the image of $v$ in $L(\mu)$ is non-zero. From the proof of Theorem  \ref{typical} it follows that $\mu\in P^+(X)$, $\mu\preceq\lambda$ and  $\chi_{\lambda}=\chi_{\mu}$. Since $L(\lambda)$ is spherically typical this implies that $\lambda =\mu$.
Thus the image of $v$ under natural homomorphism $\varphi : K(\lambda)\longrightarrow L(\lambda)$ is not zero: $\varphi(v) \ne 0.$

Let us consider the projective cover $P(\lambda)$ of $L(\lambda)$ (see e.g. \cite{Zou}) and prove that it is generated by a $\frak k$-invariant vector.   


Since $K(\lambda)$ is projective as $\frak k$-module (see Theorem \ref{Kac}), there exists a $\frak k$-invariant vector $\omega \in P(\lambda)$ such that $\psi(\omega)\ne 0$ under natural homomorphism $\psi: P(\lambda)\longrightarrow L(\lambda).$  

Let $N\subset P(\lambda)$ be the $\frak g$-submodule generated by $\omega$. Since $\psi(\omega) \ne 0$ we have that $N$ is not contained in $Ker (\psi)$, which is known to be the only maximal submodule of $P(\lambda)$ \cite{Zou}. Therefore $N=P(\lambda)$ and  $\omega$ generates $P(\lambda)$ as $\frak g$-module.  

 Now let us construct the map 
 $$
 \Psi : (P(\lambda)^*)^{\frak k}\longrightarrow  \frak A_{n,m}(\chi).
 $$
 Let $l\in P(\lambda)^*$  be a $\frak k$-invariant linear functional on $P(\lambda)$ and define $\Psi(l)$ as the restriction of $\phi_{\omega,l}=l(x\omega)$ on to $S(\frak a).$ 
 
 Let us show that $\Psi(l)\in \frak A_{n,m}(\chi).$ Similarly to the proof of Proposition \ref{restweyl} we can claim that $\Psi(l)$ has a form
$$
\Psi(l)=\sum_{\mu  \in M}c_{\mu}e^{\mu}(x), \,\, x \in U(\frak a),
$$
where the sum is taken over some finite subset of $M \subset P(X)$. Let  $\alpha$ be a root of $\frak g$ and 
$X_\alpha, X_{-\alpha} \in \frak g$ be the corresponding root vectors. The product $X_\alpha X_{-\alpha} \in U(\frak g)$ commutes with $\frak a$, so the restriction of 
$X_\alpha X_{-\alpha}\phi_{\omega,l}$ has a similar form with the same set $M.$ 
From (\ref{dalpha}) it follows that $\partial_\alpha \Psi(l)$ is divisible by $e^{2\alpha}-1$, which implies the claim.

From Theorem \ref{radial} it follows that $\Psi: (P(\lambda)^*)^{\frak k}\longrightarrow  \frak A_{n,m}(\chi)$ is a homomorphism of $Z(\frak g)$-modules.
 
 Now we are going  to prove  that $\Psi$ is an isomorphism. Let us prove first that $\Psi$ is injective. 
 Suppose that  $\Psi(l)=0$. This means that the restriction of $\phi_{\omega,l}$ on $S(\frak a)$ is zero. Therefore by Theorem \ref{radial}  $\phi_{\omega,l}(x)=l(x\omega)=0$ for all $x\in U(\frak g)$.  Since $\omega$ generates $P(\lambda)$ as $\frak g$-module we have that $l=0$ and thus $\Psi$ is injective.
 
 In order to prove that $\Psi$ is surjective it is enough to show that the dimension of $(P(\lambda)^*)^{\frak k}$ is not less than  the dimension of  $\frak A_{n,m}(\chi)$. It is known that $P(\lambda)$ has the Kac flag (see \cite{Zou}). Let $n_{\lambda,\mu}$ be the multiplicity of $K(\mu)$ in the Kac flag of $P(\lambda)$. Since Kac modules are projective as $\frak k$-modules we have 
 $$
 P(\lambda)=\bigoplus_{\mu\in Y}n_{\lambda,\mu}K(\mu),
 $$
 where $Y$ is a finite subset of highest weights $\mu$ of $\frak g$ such that $n_{\lambda,\mu}>0.$
 Therefore by part 3 of Theorem \ref{Kac}
 $$
 P(\lambda)^*=\bigoplus_{\mu\in Y}n_{\lambda,\mu}K(2\rho_1-w_0\mu)
 $$
 as $\frak k$-modules. Hence by the same Theorem we have 
 $$
 \dim (P(\lambda)^*)^{\frak k}=\sum_{\mu\in Y\cap P^+(X)}n_{\lambda,\mu}.
 $$
 We claim that $Y\subset P^+(X)$. Indeed 
 by BGG duality for classical Lie superalgebras of type I (see \cite{Zou}) we have $n_{\lambda,\mu}=m_{\mu,\lambda},$ where $m_{\mu,\lambda}$ is the multiplicity of $L(\lambda)$ in the Jordan--H\"older series of $K(\mu)$. So if $n_{\lambda,\mu}>0$ then $L(\lambda)$ is a subquotient  of $K(\mu)$. This means that there are submodules $M\supset N$ in $K(\mu)$ such that $L(\lambda)=M/N$. Therefore we have natural homomorphism
 $
\varphi: K(\lambda)\longrightarrow M/N.
 $ 
As we have just seen the image $\varphi(v)$ of $\frak k$-invariant vector $v \in K(\lambda)$ is not zero.  Since $K(\lambda)$ is projective as $\frak k$-module we can lift previous homomorphism to a homomorphism $K(\lambda)\longrightarrow M$ and the image of vector $v$ is a non-zero $\frak k$-invariant vector in $M\subset K(\mu)$. Therefore as before $\mu\in P^+(X)$, so $Y\subset P^+(X)$.
Thus we have
$$
 \dim (P(\lambda)^*)^{\frak k}=\sum_{\mu\in Y}n_{\lambda,\mu}\ge |Y|.
 $$ 
 
 By  Proposition \ref{max}  the dimension of $\frak A_{n,m}(\chi)$ is equal to the number of $\tau \in X_{n,m}^+$ such that $\chi_{\tau}=\chi_{\nu},$ or equivalently, to the number $|E|$ of the elements in the equivalence class $E.$ Let us show that $E=Y.$
 
 It is obvious that $Y \subseteq E.$ To prove that $E \subseteq Y$ consider $\mu \in E$ and corresponding Kac module $K(\mu).$ By Theorem \ref{Kac} it is spherical. Hence there exists its sub-quotient, which is spherical and irreducible. Since $\lambda$ is spherically typical this sub-quotient must be isomorphic to $L(\lambda).$ By BGG duality $\lambda \in Y.$
Thus  $\Psi$ is an isomorphism.

Now let us prove the uniqueness of $P.$ If $P(\lambda)$ and $P(\mu)$ satisfy the conditions of the theorem then they have the same central characters and from spherical typicality it follows that $\lambda=\mu$. Theorem is proved.
 \end{proof}
 
 \begin{corollary} Any finite-dimensional generalised eigenspace of $\frak D_{n,m}$ in $\frak A_{n,m}$ contains at least one zonal spherical function on $X$, corresponding to an irreducible spherically typical $\mathfrak g$-module. 
\end{corollary}

\begin{remark} Since $Y=E$ and all $n_{\lambda,\mu}=1$ as another corollary we have an effective description of the Kac flag of the projective cover $P(\lambda)$ in the spherically typical case (which may have any degree of atypicality in the sense of \cite{Brun}, see above). Our description is  equivalent to Brundan-Stroppel  algorithm \cite{Brun, Brund} in this particular case. 
\end{remark}

 \section{Zonal spherical functions for $X=(\mathfrak{gl}(1,2), \mathfrak{osp}(1,2))$.} 
 
 Let us illustrate this in the simplest example $m=1,n=1$, corresponding to the symmetric pair $X=(\mathfrak{gl}(1,2), \mathfrak{osp}(1,2)).$
 
 The corresponding algebra $\frak A_{1,1}$ consists of the Laurent polynomials 
 $f\in\Bbb C[x, x^{-1}, y, y^{-1}]$, satisfying the quasi-invariance condition
\begin{equation}
\label{quasi11}
(\partial_x+\frac12 \partial_y) f\equiv 0
\end{equation}
on the line $x=y$, where 
$
\partial_x =x \frac{\partial}{\partial x}, \, \, \partial_y =y \frac{\partial}{\partial y}.
$ 
Writing $f=\sum_{i,j \in \mathbb Z}a_{i,j}x^iy^j,$ where only finite number of coefficients are non-zero, 
we can write the quasi-invariance conditions as an infinite set of linear relations 
$$
\sum_{i+j=l}(2i+ j)a_{i,j}=0, \quad l \in \mathbb Z.
$$
Note that the algebra $\frak A_{1,1}$ is naturally $\mathbb Z$-graded by the degree 
defined for the Laurent monomial $x^iy^j$ as $i+j,$ so we have one linear relation in each degree.

The radial part of the Laplace-Beltrami operator (\ref{radpart}) in these coordinates has the form
\begin{equation}
\label{rad11}
\mathcal L_2=\partial_x^2-\frac12 \partial_y^2 - \frac{x+y}{x-y}(\partial_x+\frac12 \partial_y).
\end{equation}
It commutes with the grading (momentum) operator
$
\mathcal L_1=\partial_x+ \partial_y,
$
but in contrast with the usual symmetric spaces these two do not generate the whole algebra of the deformed CMS integrals $\frak D_{1,1}:$ 
one has to add the third order quantum integral
\begin{equation}
\label{radord3}
\mathcal L_3=\partial_x^3 + \frac{1}{4} \partial_y^3 -\frac{3}{2}\frac{x+y}{x-y} (\partial_x^2-\frac{1}{4} \partial_y^2)
+\frac{3}{4}\frac{x^2+4xy+y^2}{(x-y)^2}(\partial_x + \frac{1}{2} \partial_y).
\end{equation}
To describe the corresponding spectral decomposition let us introduce the functions
\begin{equation}
\label{fun11}
\varphi_{ij}=x^iy^j-\frac{2i+j}{2i+j-1}x^{i-1}y^{j+1}, \quad 2i+j\neq 1,
\end{equation}
\begin{equation}
\label{fun112}
\psi_{i}=x^{i+1}y^{-1-2i}+x^{i-1}y^{1-2i}, \quad i\in \mathbb Z.
\end{equation}
One can easily check that they satisfy the quasi-invariance conditions 
and form a basis in the algebra $\frak A_{1,1}$.
Denote also $\varphi_{ij}$ with $2i+j=0$ as $\varphi_i$:
$$
\varphi_i=x^i y^{-2i}, \, i\in \mathbb Z.
$$

\begin{lemma}\label{jord11} The generators $\mathcal L_i, \, i=1,2,3$ of algebra  $\frak D_{1,1}$ act 
in the basis (\ref{fun11}, \ref{fun112}) as follows
$$
\mathcal L_1 \varphi_{ij}=(i+j)\varphi_{ij}, \quad
\mathcal L_2\varphi_{ij}=\lambda_{ij}\varphi_{ij}, \quad \lambda_{ij}=i(i-1)-\frac12 j(j+1),
$$
$$
\mathcal L_3 \varphi_{ij}=\mu_{ij}\varphi_{ij}, \quad \mu_{ij}=i^3+\frac{1}{4} j^3 -\frac{3}{2} (i^2-\frac{1}{4} j^2)+\frac{3}{4} (i+\frac{1}{2}j), \quad 2i+j\neq 1,
$$
$$
\mathcal L_1 \psi_i= -i \psi_i, \quad \mathcal L_1 \varphi_i= -i \varphi_i, \quad \mathcal L_2\psi_{i}=-i^2 \psi_i-\varphi_i, \quad
\mathcal L_2\varphi_{i}=-i^2 \varphi_i, 
$$
$$
\mathcal L_3\psi_{i}=-i^3 \psi_i-3 \varphi_i, \quad
\mathcal L_3\varphi_{i}=-i^3 \varphi_i, \quad i \in \mathbb Z. 
$$
\end{lemma}

Let $W_i=< \varphi_i, \psi_i>$ be the linear span of $\varphi_i$ and $\psi_i.$ 
We see that $W_i$ is two-dimensional generalised eigenspace (Jordan block) for the whole algebra $\frak D_{1,1},$ while $V_{ij}=<\varphi_{ij}>$ 
with $2i+j \neq 0,1$ are its eigenspaces.
In our terminology $\varphi_{ij}$ with $i+2j\neq 1$ are the zonal spherical functions, while $\psi_i$ are the generalised zonal spherical functions 
of $X=(\mathfrak{gl}(1,2), \mathfrak{osp}(1,2)).$

\begin{thm}\label{irrz} 
The spectral decomposition of $\frak A_{1,1}$ with respect to the action of $\frak D_{1,1}$ has the form
\begin{equation}
\label{spec11}
\frak A_{1,1}=\bigoplus_{2i+j\neq 0,1}<\varphi_{ij}>\oplus \bigoplus_{i\in \mathbb Z}<\varphi_{i}, \psi_i>.
\end{equation}
\end{thm}

This is in a good agreement with the equivalence relation $\sim$ on $$\mathcal T=\{(a,b), \, a \in 2\mathbb Z, b \in \mathbb Z\},$$
where $a=2i, \, b =-j+1.$
Indeed, it is easy to check using Lemma \ref{equival} that the corresponding equivalence classes consist of one-element classes $(a,b)$ with $a\neq b, b-1$ and of two-element classes $(a,a) \sim (a-2, a-1).$ In terms of $i,j$ we have one element classes $(i,j)$ with $2i+j\neq 0,1$, corresponding to $V_{ij}=<\varphi_{ij}>$  and two-element classes $(i, -2i+1)\sim (i, -2i)$, corresponding to $W_i=< \varphi_i, \psi_i>.$

This also agrees with the representation theory. Let $L(\lambda)$ and $P(\lambda)$ be the irreducible spherical $\mathfrak g$-module  with highest weight $\lambda$ and its projective cover respectively.  In our case $\lambda=(2i, j, j), \, i,j \in \mathbb Z.$ The spherical typicality condition (\ref{star})
means that $2i+j\neq 1.$ In fact, one can show that if $2i+j=1$ the module $L(\lambda)$ is not spherical, in agreement with the spectral decomposition (\ref{spec11}).
If $2i+j\neq 0,1$ the projective cover $P(\lambda)=K(\lambda)$ coincides with the corresponding Kac module. One can check that it contains only one (up to a multiple) $\mathfrak {osp}(1,2)$-invariant vector with the corresponding spherical function
$\varphi_{ij}.$ A bit more involved calculations show that if $2i+j=0$ the projective cover $P(\lambda)$ contains two dimensional $\mathfrak{osp}(1,2)$-invariant subspace, corresponding to the generalised eigenspace $W_i.$

 \section{Concluding remarks} 
 
 Although we have considered only one series of the classical symmetric Lie superalgebras we believe that a similar relation of spectral decomposition of the algebra of the deformed CMS integrals and projective covers holds also at least for the remaining classical series (\ref{4series}). Note that the notion of spherically typical modules can be easily generalised to all these cases. Since the corresponding deformed root system is of $BC(n,m)$ type, to describe the corresponding zonal spherical functions one can use the super Jacobi polynomials \cite{SV2}.
 
The type $AI/AII$ we have considered in that sense is different since the corresponding deformed root system is of type $A(n-1,m-1).$ To describe the zonal spherical functions in this case we can use the theory of Jack--Laurent symmetric functions developed in \cite{SV5, SV3}.

Recall that such functions $P^{(k,p_0)}_{\alpha}$ are certain elements of $\Lambda^{\pm}$ labelled by bipartitions $\alpha=(\lambda,\mu)$, where $\Lambda^{\pm}$ is freely generated by $p_a$ with $a \in \mathbb Z \setminus \{0\}$ and variable $p_0$ is considered as an additional parameter \cite{SV5}.
 
 In \cite{SV3} we considered the case of special parameters $p_0=n+k^{-1}m$ with natural $m,n.$
In that case the spectrum of the algebra of quantum CMS integrals acting on $\Lambda^{\pm}$ is not simple. For generic $k$ we showed that any generalised eigenspace has dimension $2^r$, which coincides with the number of elements in the corresponding equivalence class of bipartitions. This equivalence can be described explicitly in terms of geometry of the corresponding Young diagrams $\lambda, \mu$ (see \cite{SV3}).
  In each equivalence class $E$ there is only one bipartition $\alpha$ such that the corresponding Jack-Laurent symmetric function $P^{(k,p_0)}_{\alpha}$ is regular at $p_0=n+k^{-1}m.$ At such $p_0$ there is a natural homomorphism
 $$
\varphi_{n,m}: \Lambda^{\pm} \rightarrow \frak A_{n,m}(k),
 $$
 sending $p_a$ to the deformed power sum 
 $$
p_a(x,y,k)= x_1^a+\dots +x_n^a+k^{-1} (y_{1}^a+\dots + y_{m}^a), \,\, a \in \mathbb Z.
 $$
The image of the corresponding function $\varphi_{n,m}(P^{(k,n+k^{-1}m)}_{\alpha}) \in \frak A_{n,m}(k)$ is an eigenfunction of the algebra $\frak{D}_{n,m}(k)$ of the deformed CMS integrals, so its specialisation at $k=-1/2$ (provided it exists) determines a zonal spherical function for $X=(\mathfrak{gl}(n,2m), \mathfrak{osp}(n,2m)).$ 
A natural question is whether this relation can be extended to an isomorphism of the corresponding generalised eigenspaces.

We would like to mention also that Brundan and Stroppel \cite{Brund} showed that the algebra of the endomorphisms of a projective indecomposable module over general linear supergroup is isomorphic to
$$
 \mathfrak A_r=\Bbb C[\varepsilon_1,\varepsilon_2,\dots,\varepsilon_r]/(\varepsilon_1^2,\,\varepsilon_2^2,\dots,\varepsilon_r^2).
 $$
We believe that using this and our main theorem it is possible to describe in a similar way the action of the algebra $\frak{D}_{n,m}$ in its generalised eigenspace, which in particular would imply that it contains only one zonal spherical function. This would be in a good agreement with the results of \cite{SV3}.

 \section{Acknowledgements}
 
 We are grateful to Professor A. Alldridge for attracting our attention to the preprint \cite{AldSch} and helpful comments on the first version of this paper.

This work was partially supported by the EPSRC (grant EP/J00488X/1). ANS is grateful to the Department of Mathematical Sciences of Loughborough University for the hospitality during the autumn semesters in 2012-14.

\end{document}